\documentclass[11pt,oneside,a4paper]{article}
\usepackage[top=1 in, bottom=1 in, left=0.85 in, right=0.85 in]{geometry}

\usepackage{amsmath} 
\usepackage{amsthm} 
\usepackage{amssymb}
\usepackage{tabularx}
\usepackage{graphicx} 
\usepackage{subcaption}
\usepackage{algorithm}
\usepackage{graphicx}
\usepackage{epstopdf}
\usepackage{enumerate}
\usepackage{stfloats}
\usepackage{xcolor}
\usepackage{thmtools}
\usepackage{thm-restate}
\usepackage{multirow}
\usepackage[hidelinks]{hyperref}

\newtheorem{theorem}{Theorem}
\newtheorem{lemma}{Lemma}

\newtheorem{remark}{Remark}

\newcommand{\delete}[1]{}

\title{Large-System Insensitivity of Zero-Waiting Load Balancing Algorithms}
\author{Xin Liu \\ShanghaiTech University \\ liuxin7@shanghaitech.edu.cn
\and Kang Gong \\University of Michigan, Ann Arbor \\ kanggong@umich.edu
\and Lei Ying \\University of Michigan, Ann Arbor \\ leiying@umich.edu}
\date{}

\begin{document}

\maketitle

\begin{abstract}
This paper studies the sensitivity (or insensitivity) of a class of load balancing algorithms that achieve asymptotic zero-waiting in the sub-Halfin-Whitt regime \cite{LiuYin_20}, named LB-zero. Most existing results on zero-waiting load balancing algorithms assume the service time distribution is exponential.  This paper establishes the {\em large-system insensitivity} of LB-zero for jobs whose service time follows a Coxian distribution with a finite number of phases. This result suggests that LB-zero achieves asymptotic zero-waiting for a large class of service time distributions,
which is confirmed in our simulations. To prove this result, this paper develops a new technique, called ``Iterative State-Space Peeling'' (or ISSP for short). ISSP first identifies an iterative relation between the upper and lower bounds on the queue states and then proves that the system lives near the  fixed point of the iterative bounds with a high probability. Based on ISSP, the steady-state distribution of the system is further analyzed by applying Stein's method in the neighborhood of the fixed point. ISSP, like state-space collapse in heavy-traffic analysis, is a general approach that may be used to study other complex stochastic systems. 
\end{abstract}

\section{Introduction}
Zero-waiting load balancing refers to a load balancing algorithm under which a job is routed to an idle server to be processed immediately upon its arrival. The problem has become increasingly important as the amount of modern machine learning (ML) and artificial intelligence (AI) applications running on large-scale data centers explodes. While increasing the number of servers and the processing speed of each server is a critical step to meet the increasing demand, the design of load balancing algorithms that can efficiently utilize available resources to minimize or even eliminate the waiting time of incoming jobs is equally important,  especially when a minor increase of latency (e.g. 100 milliseconds) can lead to a significant drop in a cloud-computing provider's revenue (7\% drop  in sales according to a recent Akamai report \cite{Aka_17}).

Significant processes have been made over the past few years on understanding achieving asymptotic zero-waiting (as the system size approaches infinity) in a large-scale data center with distributed queues, including the classic supermarket model \cite{EscGam_18,Bra_20,Sto_15,GupWal_19,BanMuk_19,BanMuk_20,MukBorvan_18,LiuYin_20,LiuYin_21,LiuYin_18_Infocom,LiuGonYin_21,ZhaBanMuk_21,Bra_22}, models with data locality \cite{WenZhoSri_20, DanDeb_21} and models where each job consists of parallel tasks \cite{WenWan_20,WanXieHar_21,HonWan_21}, etc. 

However, almost all these results assume exponential service time distributions. While each of these results \cite{EscGam_18,Bra_20,Sto_15,GupWal_19,DebankSemLee_18,MukBorvan_18,LiuYin_20,LiuYin_21,LiuYin_18_Infocom,LiuGonYin_21,WenZhoSri_20, DanDeb_21,WenWan_20,WanXieHar_21,HonWan_21,ZhaBanMuk_21,Bra_22} provided important insights of achieving zero-waiting in a practical system, theoretically, it is not clear whether these principles hold for general service times. This is a very important question to answer because it is well-known that service time distributions in real-world systems are not exponential. Understanding a queueing system's performance with general service time remains one of the most important and intensively studied problems in stochastic networks. A concept that excited many theorists in the area is ``insensitivity'' \cite{Dav_81}. A queueing system is called insensitive if the steady-state distribution of queue lengths  is invariant to the service time distribution. Therefore, any conclusion drawn from exponential service time distributions can be applied to general service time distributions. A result that is insensitive is robust and is expected to be widely applicable in practical systems. Unfortunately, insensitivity results are rare and often hold only under some special queueing disciplines such as processor sharing (PS) \cite{Dav_81, BonJonPro_04, KieVan_21}. One of the reasons is that insensitivity, while appealing, is a very strong notion of ``robustness''. It requires the steady-state distribution under a general service time distribution to be {\em exactly the same} as that under the exponential distribution.  Some recent studies started to relax it to weaker notions such as insensitivity in the heavy-traffic regime \cite{WanMagSri_21} or the large-system regime \cite{BraLuPra_12,Sto_15}, i.e. insensitivity in the limiting regimes. In the light of these recent developments, this paper addresses the following important question: 

\centerline{\em Are the zero-waiting algorithms insensitive and if so, in which notion of insensitivity? } 

\subsection{Main Contributions}
This paper provides some positive answers to the question above. First, it is well known that most of zero-waiting algorithms, such as join-the-shortest-queue (JSQ) \cite{Win_77} and join-the-idle-queue (JIQ) \cite{LuXieKli_11}, are not insensitive (according to its original definition). However, we prove that in the sub-Halfin-Whitt regime, LB-zero identified in \cite{LiuYin_20} in fact achieves asymptotic zero-waiting for jobs whose service time follows a Coxian distribution with a finite number of phases. This result establishes the {\em large-system insensitivity of LB-zero for Coxian service time distributions with a finite number of phases.} Since the Coxian family is dense in the class of positive-valued distributions, our result strongly suggests 
a load balancing algorithm in the LB-zero family will be able to minimize unnecessary waiting in large-scale data centers for a large class of job size or service times  distributions. Our simulations further confirm it.

To prove this result, this paper develops a new technique, called ``iterative state-space peeling'' (or ISSP for short). ISSP first identifies an iterative relation between upper and lower bounds on the queue states. 
Then by iteratively ``peeling off'' the low-probability states, it proves that the system ``lives'' near  the fixed point of the iterative bounds with a high probability. Based on ISSP, the steady-state distribution of the system can be further analyzed by using Stein's method in a small neighborhood of the fixed point. ISSP, like the state-space collapse in the heavy-traffic analysis, is a general technique that may be used to study other complex stochastic systems, e.g. large-system insensitivity of load balancing algorithms for other models like those studied in \cite{DebankSemLee_18,WenWan_20,WenZhoSri_20,WanXieHar_21,WenSri_22}. 

We remark that this paper does not establish the large-system insensitivity for an {\em arbitrary} service time distribution, for which we need to show that our results continue to hold for a large but finite $N$ when the number of phases of the Coxian distribution goes to  infinity. This requires an interchange limits arguments or a continuity argument and is an interesting open problem. It is also worth mentioning that a Coxian representation of a probability distribution is non-unique. The choice of the Coxian representation is out the scope of this paper.

\subsection{Related Work}
Steady-state analysis of distributed queueing systems has been an active research topic since the seminal work on power-of-two-choices \cite{Mit_96,VveDobKar_96}. 
The most popular approach to study a large-scale distributed queueing system is the mean-field approach where the system is approximated using a deterministic dynamical system (a set of ordinary differential equations), called a mean-field model. In the large-system limit (as the number of servers approaches infinity), the steady-state of the stochastic system can often be shown to converge to the equilibrium point of the mean-field model using the interchange of limits (e.g. \cite{VveDobKar_96,XieDonLu_15,YinSriKan_15}) or Stein's method (e.g. \cite{Yin_16,Gas_17}).

While most studies on this topic assume exponential service time distributions for tractability, the approach has been used to study non-exponential service time distributions theoretically or numerically (see. e.g. \cite{Mit_96,BraLuPra_12,Sto_15,AghLiRam_17, VasMukMaz_17, HelVan_18,Van_19,KieVan_21}).
For example, when non-exponential service time distributions has a decreasing hazard rate (DHR), the system often exhibits  a monotonicity property such that the system starting from  the empty state is dominated by the system starting from any other state. Leveraging this monotonicity, \cite{BraLuPra_12} proved the convergence of power-of-$d$-choices and  \cite{Sto_15} proved the convergence of JIQ to the corresponding mean-field limit, respectively. 
Recently, \cite{Van_19} studied load balancing policies under the hyper-exponential service time distribution in the light traffic regime (or in a critical traffic regime). By transforming the hyper-exponential distribution to a Coxian distribution with DHR, the monontoncity property holds in a partial order and the global stability of the mean-field model was established. \cite{KieVan_21} studied Po$d$ with PS servers for a hyper-exponential distribution of order $2$ in the light traffic regime. They also established the global stability result in a spirit similar to \cite{Van_19}. We note \cite{KieVan_21} considered  Po$d$ where $d$ is a constant independent of the number of servers so the system has a non-vanishing delay in the large-system limit. For the Po$d$ algorithm in the LB-zero family, $d$ is a function of $N$ and the algorithm achieves asymptotic zero waiting in the heavy-traffic regime without the DHR assumption.

Without DHR, the results are very limited. \cite{BraLuPra_12} proved the convergence of join the least loaded of $d$ queues (LL(d)) for general service time distributions and that of Po$d$ when the load of the system is small (less than $1/4$).  \cite{FosSto_17} proved the asymptotic optimality of JIQ under general service time distributions when the normalized load is less than $0.5.$ Since JIQ is an LB-zero policy, our result confirms the conjecture made in \cite{FosSto_17} that JIQ is asymptotically optimal for any load less than one, not just less than $0.5.$  Another significant result is \cite{BonJonPro_04}, which identifies a set of policies that are insensitive in many-server load-balancing systems and are optimal in the class of insensitive load balancing algorithms. The asymptotic blocking probability of this class of insensitive algorithms in a finite buffer system was later studied in  \cite{JonBal_16}. Many LB-zero such as JSQ and Po$d$ algorithms are sensitive, so the results in \cite{BonJonPro_04,JonBal_16} do not apply. We also note that the waiting probability in our paper includes both blocking and being queued in the buffer, so our result implies asymptotic zero blocking of LB-zero in the sub-Halfin-Whitt regime.

\cite{LiuGonYin_21} is the work most related to this paper, which considers the Coxian-2 distribution and shows that LB-zero achieves asymptotic zero-waiting in the sub-Halfin-Whitt regime. 
Inspired by \cite{LiuGonYin_21}, this paper develops the ISSP technique for general Coxian distributions with a finite number of phases and establishes its large-system insensitivity. 
We remark that \cite{LiuGonYin_21} utilized a key property of Coxian-$2$ service time distribution that a job in the first phase (phase-$1$) either departs or enters the last phase (phase-$2$) immediately, which does not hold under a general Coxian distribution which may have many phases. 

In terms of the proof, each step in ISSP utilizes the tail bound in \cite{BerGamTsi_01} to ``peel off'' a low probability subspace. The tail bound is based on the Lyapunov drift analysis, and is a critical step to prove state-space collapse in the traditional heavy traffic regime with a fixed number of servers  (see e.g., \cite{ErySri_12,MagSri_16,WanMagSri_21}). The key difference is that \cite{ErySri_12} utilizes the tail bound only once while ISSP repeatedly utilizes the tail bound guided by an iterative relation between the upper and lower bounds.

\section{Model and State Representation}
We consider a many-server system with $N$ homogeneous servers, where job arrival follows a Poisson process with rate $\lambda N$ with $\lambda = 1 - N^{-\alpha}, 0<\alpha <0.5$, i.e. the system is in the sub-Halfin-Whitt regime. We assume the service times follow the Coxian distribution with $M$ phases
as shown in Figure~\ref{fig:coxM}, where $\mu_m>0$ is the rate a job finishes phase $m$ when in service and $0\leq p_i<1, 1 \leq i < M$ is
the probability that a job enters phase $i+1$ after completing  phase $i$ and $p_M=1.$ Note we assume $\lambda = 1 - N^{-\alpha}$ for the ease of exposition and our results can be easily extended to the case that $\lambda = 1 - \beta N^{-\alpha}$ with any positive constant $\beta >0$ independent with $N.$ As convention, we define $\sum_{i=a}^b x_i = 0$ if $a > b$ and $\prod_{i=a}^b x_i = 1$ if $a > b$ for the series $\{x_i, i \geq 1\}.$ 
\begin{figure}[!htbp]
  \centering
  \includegraphics[width=5.3in]{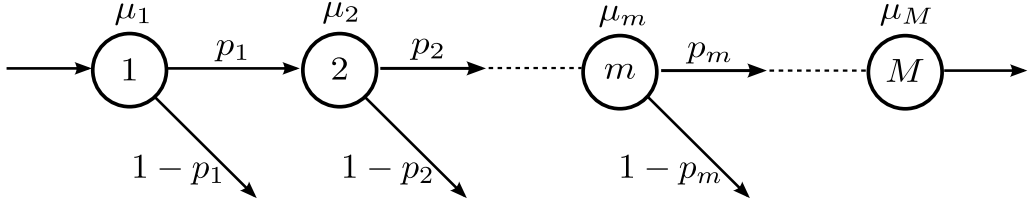}
  \caption{The Coxian-$M$ Distribution: $\mu_m$ is the service rate in phase $m$ and $p_m$ is the probability entering phase $m+1$ after finishing service in phase $m.$}
  \label{fig:coxM}
\end{figure}

Without loss of generality, we normalize the mean service time to be one, i.e. $$\sum_{m=1}^{M}v_m=1~~~\text{with}~~~v_m=\frac{\prod_{i=1}^{m-1}p_i}{\mu_m},$$
where $v_m$ is viewed as the average time spent in phase $m$ for a job. Given the unit service rate, $\lambda$ is the normalized load of the system and $\lambda v_m$ is the normalized load of jobs in phase $m.$ 

Taking Coxian-$3$ distribution as an example (see Figure~\ref{model-coxianM}), a job is colored in black if it waits in the buffer, and colored in light red, blue, and green when it is in phase $1,$ $2$ and $3$, respectively. Jobs are served with the FIFO discipline and we assume each server has a buffer of size $b-1,$ so can hold at most $b$ jobs ($b-1$ in the buffer and one in service). The assumption of finite buffer is imposed due to a technical reason and will be explained later in the paper. Relaxing the finite-buffer assumption is not trivial technically but we conjecture our results hold without this assumption.

\begin{figure}[!htbp]
  \centering
  \includegraphics[width=4.0in]{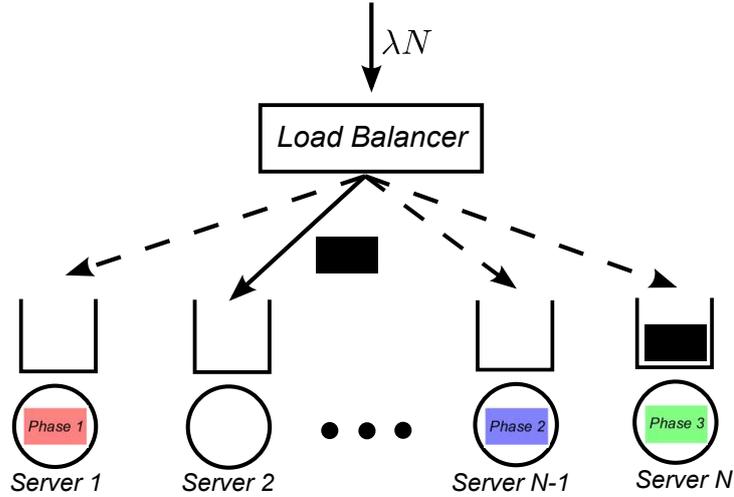}
  \caption{~Load Balancing in Many-Server Systems with Coxian-$3:$ jobs colored in black, light red, blue, and green represent jobs in the buffer, in phase $1,$ $2,$ and $3,$ respectively.}
  \label{model-coxianM}
\end{figure}

To represent the system, define $Q_{j,m}(t)$ ($m=1,2,\cdots, M$) to be the fraction of servers which have $j$ jobs at time $t$ and the one in service is in phase $m$. Because an idle server does not have a phase,  we define $Q_{0,1}(t)$ to be the fraction of servers that are idle at time $t$ and $Q_{0,m}(t)=0, \forall 2 \leq m \leq M$ for convenience. We stack $Q_{j,m}(t)$ to a matrix $Q(t)\in R^{b\times M}$ such that the $(j,m)$th entry of the matrix is $Q_{j,m}(t).$ 
We further define $S_{i,m}(t) = \sum_{j \geq i} Q_{j,m}(t)$ and $S_{i}(t) = \sum_{m=1}^M S_{i,m}(t).$ Therefore,  $S_{i,m}(t)$ is the fraction of servers which have at least $i$ jobs and the job in service is in phase $m$ at time $t$ and $S_i(t)$ is the fraction of servers with at least $i$ jobs at time $t.$ Stack $S_{j,m}(t)$ to be a matrix $S(t)$ such that the $(j,m)$th entry of the matrix is $S_{j,m}(t).$ Since $Q(t)$ and $S(t)$ have a one-to-one mapping, we focus on $S(t)$ throughout the paper. We consider load balancing policies which dispatch jobs to servers based on $S(t)$ and under which the finite-state CTMC $\{S(t), t\geq 0\}$ is irreducible, and so it has a unique stationary distribution. This includes well-known load balancing policies such as JSQ \cite{Win_77,EscGam_18,Bra_20}, JIQ \cite{LuXieKli_11,Sto_15}, I1F \cite{GupWal_19} and Po$d$ \cite{Mit_96,VveDobKar_96}

Let $Q_{j,m}$ be a random variable that has the distribution of $Q_{j,m}(t)$ at steady state. Correspondingly, define $S_{i,m}= \sum_{j \geq i} Q_{j,m}$ and $S_{i} = \sum_{m} S_{i,m}.$ In other words,  $S_{i,m}$ is the fraction of servers which have at least $i$ jobs and the job in service is in phase $m$ and $S_i$ is the fraction of servers with at least $i$ jobs, both at steady state. Consider a system with $10$ servers and Coxian-3 service time distribution. A realization of state representation $S_{i,m}$ is shown in Figure~\ref{model-state} and Table \ref{tab:exa}. Define $S\in R^{b\times M}$ to be a matrix such that the $(i,m)$th entry is  $S_{i,m}$ and $s \in \mathbb R^{b\times M}$ to be a realization of $S.$ Define $\mathcal S^{(N)}$ to be a set of $s$ as follows
\begin{align}
\mathcal{S}^{(N)}  = \left\{ s ~\left|~ 1\geq s_{1,m}\geq \cdots \geq s_{b,m}\geq  0, ~ 1\geq \sum_{m=1}^Ms_{1,m}, ~ Ns_{i,m}\in \mathbb{N}, ~ \forall i, m\right.\right\} \label{state_S},
\end{align}
i.e., $\mathcal S^{(N)}$ is the set of all possible $s$ in a system with $N$ servers. 

\begin{figure}[!htbp]
  \centering
  \includegraphics[width=6.1in]{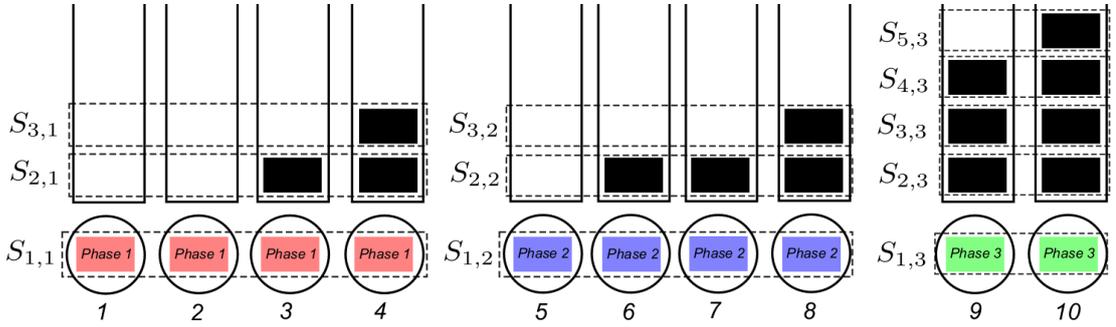}
  \caption{An example of the realization of $S_{i,m}$ in a system with $10$ servers and Coxian-3 service time distribution}
  \label{model-state}
\end{figure}

\begin{table}[!htbp]
\centering
\begin{tabular}{ccc|ccc|ccccc}
$Q_{1,1}$ & $Q_{2,1}$ & $Q_{3,1}$ & $Q_{1,2}$ & $Q_{2,2}$ & $Q_{3,2}$ & $Q_{1,3}$ & $Q_{2,3}$ & $Q_{3,3}$ & $Q_{4,3}$ & $Q_{5,3}$ \\ \hline
0.2       & 0.1       & 0.1       & 0.1       & 0.2       & 0.1       & 0.0       & 0.0       & 0.0       & 0.1       & 0.1       \\ \hline
$S_{1,1}$ & $S_{2,1}$ & $S_{3,1}$ & $S_{1,2}$ & $S_{2,2}$ & $S_{3,2}$ & $S_{1,3}$ & $S_{2,3}$ & $S_{3,3}$ & $S_{4,3}$ & $S_{5,3}$ \\ \hline
0.4       & 0.2       & 0.1       & 0.4       & 0.3       & 0.1       & 0.2       & 0.2       & 0.2       & 0.2       & 0.1      
\end{tabular}
\caption{~The corresponding values of $Q_{i,m}$ and $S_{i,m}$ in Figure~\ref{model-state}: The system in Figure \ref{model-state} includes four busy servers that are serving jobs in phase $1,$ including two servers without any waiting jobs $(Q_{1,1}=0.2),$ one server with one waiting job $(Q_{2,1}=0.1),$ and one server with two waiting jobs $(Q_{3,1}=0.1);$ four busy servers that are serving jobs in phase $2,$ including one server without any waiting jobs $(Q_{1,2}=0.1),$ two server with one waiting job $(Q_{2,2}=0.2),$ and one server with two waiting jobs $(Q_{3,2}=0.1);$  and two busy servers that are serving jobs in phase $3,$ including one server with three waiting jobs $(Q_{4,3}=0.1)$ and one server with four waiting jobs $(Q_{5,3}=0.1).$ \label{tab:exa}}
\end{table}

\section{Main Results}
Before introducing the main results, we first define several constants that will be used throughout the paper: 
\begin{align*}
a_m =& \frac{\mu_m}{p_1\mu_1 + \mu_m}& 2\leq m\leq M \\
b_m =& (1-a_m)\left(1 + \sum_{r=m+1}^M \frac{v_{r}}{v_{1}}\right) - \frac{a_mv_m}{v_1}&2\leq m\leq M\\  
\xi=&\sum_{m=2}^{M} b_m \prod_{j=m+1}^M a_j&\\ 
c_m =& 5(1-a_m)\sum_{r=m+1}^M (r-1)v_r + 5a_m \sum_{r=2}^{m-1} \frac{\mu_r v_r}{\mu_m} + 5(m-2)a_m v_m + 5 - a_m & 2\leq m \leq M\\
C_M =& \sum_{m=2}^{M} c_m\prod_{j=m+1}^M a_j.&
\end{align*}
These constants are positive constants and $0<\xi<1.$ The proof can be found in Appendix \ref{app:constants}. Their values depend on the Coxian-$M$ distribution but are independent of $N.$

We further define $A_1(s)$ to be the probability that an incoming job is routed to a busy server conditioned on that the system is in state $s \in \mathcal{S}^{(N)};$ i.e.  $$A_1(s)= \mathbb P \left(\left. \text{an incoming job is routed to a busy server} \right| {S(t)=s}\right).$$  

We now consider the set of zero-waiting load balancing policies, named as {\em LB-zero}, 
\begin{align}\hbox{\bf LB-zero}:\ \Pi=\left\{\pi~ \left|~ \text{Under policy}~ \pi, A_1(s)\leq \frac{1}{\sqrt{N}} ~\text{for any}~ s\in{\mathcal{S}^{(N)}} \right.\right.  \left.\left. \text{such that}~s_1\leq 1-\frac{1}{N^\alpha\log N}\right.\right\}. \nonumber
\end{align} Note that this class of policies is similar to the one considered in \cite{LiuYin_20}. Several well-known policies satisfy this condition, as summarized in Table \ref{tab:lb in pi}.

\begin{table}[H] 
\begin{tabular}{|c|l|c|}
\hline
Load Balancing Policy                                                                         & \multicolumn{1}{c|}{Description}                                                                                                                                                      & Condition                                                                                                                          \\ \hline
Join-the-Shortest-Queue                                                                       & route an incoming job to the least loaded server                                                                                                                                                 & $A_1(s)=0$ for $s_{1}<1$                                               
                                             \\ \hline
Join-the-Idle-Queue                                                                           & \begin{tabular}[c]{@{}l@{}}route an incoming job to an idle server \\ if available and otherwise, to a server \\ chosen uniformly at random.\end{tabular}                               & $A_1(s)=0$ for $s_1<1$                                 
                                    \\ \hline
Idle-One-First                                                                                & \begin{tabular}[c]{@{}l@{}}route an incoming job to an idle server\\ if available; to a server with one job if available; \\ and otherwise, to a randomly selected server.\end{tabular} & $A_1(s)=0$ for $s_1<1$                                                                                                                           \\ \hline
\begin{tabular}[c]{@{}c@{}}Power-of-$d$-Choices\\ with $d\geq N^\alpha\log^2 N,$\end{tabular} & \begin{tabular}[c]{@{}l@{}} sample $d$ servers uniformly at random and \\ route the job to the least loaded \\ server among the $d$ servers.\end{tabular}                                     & \begin{tabular}[c]{@{}c@{}}For sufficiently large $N,$\\ $A_1(s)\leq \frac{1}{\sqrt{N}}$ \\ for $s_1\leq 1-\frac{1}{N^\alpha\log N}$\end{tabular} \\ \hline
\end{tabular} 
\caption{Examples of LB-Zero Policies:  Join-the-Shortest-Queue, Join-the-Idle-Queue, Idle-One-First, and Power-of-$d$-Choices with a carefully chosen $d.$} \label{tab:lb in pi}
\end{table}

To prove the large-system insensitivity of LB-zero, we first show that $S_{1,m}$ is ``close'' to $s_{1,m}^* = \lambda v_m,$ which is the normalized load from phase-$m$ of the jobs, and is also the equilibrium point of the mean-field system assuming zero-waiting (details can be found in Section \ref{sec: sys dynamics} and \ref{sec: issc}). We call $s^*$ the zero-waiting equilibrium. Theorem \ref{Thm:equilibrium} shows that at the steady-state, $S_{1,m}$ concentrates around the zero-waiting equilibrium $s_{1,m}^*$ for large $N.$ The proof of this theorem can be found in Section \ref{sec: prove thm 1}. 
\begin{theorem}[High Probability Bound]\label{Thm:equilibrium}
Define $\theta_m = \frac{6\mu_1 v_m +  5(m-1)v_m}{C}, \forall 1 \leq  m \leq M$ with $C = \sqrt{\frac{2\bar v^2\log (1/\xi)}{3M + (3M+4)\log (1/\xi)}}$ and $\bar v = \min_m v_m.$ For any LB-zero policy in $\Pi,$ the following bound holds  
$$\mathbb P\left(s_{1,m}^* + \frac{\theta_m\log N}{\sqrt{N}}\leq S_{1,m} \leq s_{1,m}^* + \frac{1}{N^\alpha} - \frac{\sum_{r\neq m}\theta_r\log N}{\sqrt{N}}\right) \geq 1-\frac{M}{N^3}$$ when $N$ satisfies
\begin{align}
 \min\left\{\sum_{m=1}^M\theta_m,\frac{C(1-\xi)}{2\mu_1 C_M}\right\} N^{0.5-\alpha} \geq \log N \geq \max\left\{\frac{2\mu_1}{1-\xi}, \frac{C}{\mu_1}\right\}. 
\end{align}
 \qed
\end{theorem}

\begin{remark}
Theorem \ref{Thm:equilibrium} shows that $S_{1,m}$ differs from  $s_{1,m}^*$ by at most $\max\left\{\frac{\theta_m\log N}{\sqrt{N}}, \frac{1}{N^\alpha}\right\}$ with a probability at least $1-M/N^3,$ which implies that the convergence the steady-state $S_{1,m}, \forall m$ to the zero-waiting equilibrium as $N\to \infty$ in probability and mean-square senses.  
To be best of our knowledge, this is the first result to establish such a  steady-state convergence for a load balancing system under Coxian-$M$ service time distributions in the heavy-traffic regime. Since the high probability bound holds for a large but finite $N,$ it also provides the rate of convergence.
\end{remark}

From Theorem \ref{Thm:equilibrium}, it is not clear whether the probability of waiting approaches zero under an LB-zero policy, which will be studied in the next theorem. Let $\mathcal W$ denote the event that an incoming job is routed to a busy server in the system, and $\mathbb P(\mathcal W)$ denote the probability of this event at steady-state.
We have the following result on the waiting probability. The proof of Theorem \ref{Thm:main} can be found in Section \ref{sec: prove thm 2}.

\begin{theorem}\label{Thm:main}
Define $w_m = (1-p_m)\mu_m,$ $w_u=\max_m w_m,$ $w_l=\min_m w_m,$ $\mu_{\max}=\max_m \mu_m,$ $\zeta = \frac{4w_ub}{w_l}[(\frac{1}{w_l}-\frac{1}{w_u})\sum_m \theta_m w_m + \frac{1}{w_l}+6],$ and $k = \frac{\sum_{m}\theta_m w_m}{w_u}+(1+\frac{w_l}{4w_ub})\zeta - \sum_m \theta_m.$ 
Under an LB-zero policy in $\Pi,$ the following result holds  
\begin{align}\mathbb P(\mathcal W)\leq \frac{1}{\sqrt{N}} + \frac{10\mu_{\max}+4}{N^{0.5-\alpha}\log N}. \label{prob: waiting}
\end{align} 
when $N$ satisfies
\begin{align}
 \min\left\{\frac{1}{2k},\sum_{m=1}^M\theta_m,\frac{C(1-\xi)}{2\mu_1 C_M}\right\} N^{0.5-\alpha} \geq \log N \geq \max\left\{\log\left(\frac{1}{\xi}\right), \frac{2\mu_1}{1-\xi}, \frac{4b}{w_l \zeta}, \frac{C}{\mu_1}, 2\right\}. \label{N-cond}    
\end{align}
\qed
\end{theorem}
\begin{remark}
Theorem \ref{Thm:main} shows the waiting probability is $O(1/N^{0.5-\alpha})$ for a large but finite $N,$ which implies the asymptotic zero waiting, i.e. $\mathbb P(\mathcal W) \to 0$ as $N\to \infty$ in the sub-Halfin-Whitt regime. This asymptotic result implies LB-zero is large-system insensitive to Coxian-$M$ distributions with a finite number of phases.
\end{remark}
Next, we will establish these two main results. We first introduce the system dynamic of LB-zero in Section \ref{sec: sys dynamics} and present  ``iterative state-space peeling'' (ISSP) in Section \ref{sec: issc}, which is used to prove that the system lives near a limiting regime in Theorem \ref{Thm:equilibrium} in Section \ref{sec: prove thm 1} and to prove the large system insensitivity in Theorem \ref{Thm:main} in Section \ref{sec: prove thm 2}. 

\section{System Dynamics}\label{sec: sys dynamics}
Define $e_{i,m} \in \mathbb R^{b\times M}$ to be a $b\times M$-dimensional matrix with the $(i,m)$th entry being $1/N$ and all other entries being zero. Furthermore, define $A_{i,m}(s)$ to be the probability that an incoming job is routed to a server with at least $i$ jobs and the job in service is in phase $m$ given the system state $s,$ i.e.
\begin{align*}
A_{i,m}(s)=\mathbb P \left( \hbox{an incoming job is routed to a server with at least $i$ jobs} \right.\\
\left. \hbox{and the job in service is in phase $m$} ~|~{S(t)}=s\right). 
\end{align*}

Given state $s$ (or the corresponding $q$) of the CTMC, each of the following three events triggers a state transition, which is illustrated individually in Figure \ref{fig:events}. 
\begin{itemize}
\item Event 1: A job arrives and is routed to a server that has $i-1$ jobs and the job in service is in phase $m$ as in the left figure in Figure \ref{fig:events}. When this occurs, $q_{i,m}$ increases by $1/N,$ and $q_{i-1,m}$ decreases by $1/N$ (note $m=1$ if $i=1$ since we define the faction of idle servers to be $q_{0,1}$). So the CTMC has the following transition: 
\begin{align*}
q \to& ~q+e_{i,m}-e_{i-1,m},\\
s \to& ~s+e_{i,m},
\end{align*}
where the transition of $s$ can be verified according to the definition $s_{i,m}=\sum_{j\geq i} q_{j,m}$ so only $s_{i,m}$ increasing by $1/N.$
This event occurs with rate $$\lambda N(A_{i-1,m}(s)-A_{i,m}(s)),$$ where $A_{i-1,m}(s)-A_{i,m}(s)$ is the probability that an incoming job is routed to a server that has $i-1$ jobs and the job in service is in phase $m.$ 

\item Event 2: A server with $i$ jobs finishes serving a job in phase $m.$ The job departs from the system without entering into the next phase as in the middle figure in Figure \ref{fig:events}. When this event occurs, $q_{i,m}$ decreases by $1/N$ and $q_{i-1,1}$ increases by $1/N,$ so the CTMC has the following transition:
\begin{align*}
q \to& ~q-e_{i,m}+e_{i-1,1},\\
s \to& ~s-\sum_{j=1}^ie_{j,m}+\sum_{j=1}^{i-1}e_{j,1},
\end{align*}
where the transition of $s$ can be verified based on the definition $s_{i,m}=\sum_{j\geq i} q_{j,m}$ so $s_{j,m}$ decreases by $1/N$ for any $j\leq i$ and $s_{j,1}$ increases by $1/N$ for any $j< i.$
This event occurs with rate $$\mu_mNq_{i,m}(1-p_m),$$
where $\mu_mNq_{i,m}$ is the rate at which a job in phase $m$ finishes the service and $(1-p_m)$ is the probability that a job finishes phase $m$ and departs from the system immediately.

\item Event 3: A server with $i$ jobs finishes serving a job in phase $m,$ and the job enters the next phase $m+1$ as shown in the right figure in Figure \ref{fig:events}. 
When this event occurs, a server in state $(i,m)$ transits to state $(i,m+1),$ so $q_{i,m}$ decreases by $1/N$ and $q_{i,m+1}$ increases by $1/N.$ Therefore, the CTMC has the following transition:
\begin{align*}
q \to& ~q-e_{i,m}+e_{i,m+1},\\
s \to& ~s-\sum_{j=1}^i e_{j,m} + \sum_{j=1}^i e_{j,m+1},
\end{align*} 
where the transition of $s$ holds because $s_{i,m}$ decreases by $1/N$ for any $j \leq i$ and $s_{j,m+1}$ increases by $1/N$ for any $j\leq i.$ 
This event occurs with rate $$\mu_mNq_{i,m}p_m,$$
where $\mu_mNq_{i,m}$ is the rate at which a job in phase $m$ finishes the service and $p_m$ is the probability that a job enters phase $m+1$ after finishing  phase $m.$ 
\end{itemize}

\begin{figure}[H]
  \centering
  \includegraphics[width=5.9in]{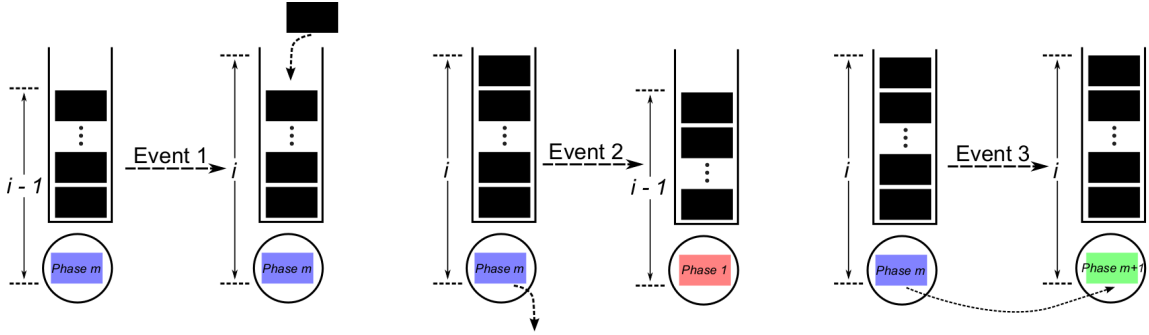}
  \caption{~Illustrations of State Transitions on ($q_{i,m}$ or $s_{i,m}$) Triggered by the Three Events: 1) a job arrives to a server with $i-1$ jobs and the job in service in phase $m;$ 2) a server with $i$ jobs finishes a job in phase $m,$ and the job departs from the system; 3) a server with $i$ jobs finishes a job in phase $m,$ and the job enters into the next phase.}
  \label{fig:events}
\end{figure}

Based on the three events above, we focus on $(S(t):t\geq 0)$ because $S(t)$ and $Q(t)$ has one-to-one mapping and the dynamics of $S(t)$ have a simpler form. Define $G$ to be the generator of CTMC $(S(t):t\geq 0).$ Given function $f: \mathcal S^{(N)} \to \mathbb{R},$ we have
\begin{align}
G f(s) = &\sum_{i=1}^{b}\sum_{m=1}^{M}\left[\lambda N(A_{i-1,m}(s)-A_{i,m}(s))(f(s + e_{i,m}) - f(s))  \right.  \label{G-Si-A}\\ 
&~~~~~\left. + (1-p_m)\mu_m Nq_{i,m}\left(f\left(s-\sum_{j=1}^ie_{j,m}+\sum_{j=1}^{i-1}e_{j,1}\right)-f(s)\right) \right. \label{G-Si-D}\\
&~~~~~\left. + p_m\mu_m Nq_{i,m}\left(f\left(s - \sum_{j=1}^ie_{j,m}+\sum_{j=1}^ie_{j,m+1}\right)-f(s)\right)\right]. \label{G-Si-change} 
\end{align}  

To understand the dynamics better, we write down the mean-filed model (MFM) according to the generator: 
\begin{align}
\dot s_{i,1} &= \lambda (A_{i-1,1}(s) - A_{i,1}(s)) + \sum_{m=2}^{M}(1-p_m)\mu_{m}s_{i+1,m}  - \mu_1 s_{i,1}, \label{mfm:s1} \\
\dot s_{i,m} &= \lambda (A_{i-1,m}(s) - A_{i,m}(s)) + p_{m-1}\mu_{m-1}s_{i,m-1} - \mu_m s_{i,m}, ~\forall m \geq 2. \label{mfm:sm} 
\end{align} 

This mean-field model is nonlinear in $s$ because $A_{i,m}(s)$ is a nonlinear function in $s,$ and its equilbirium point is difficult to calculate in general. However, suppose zero-waiting occurs, i.e., $s_{i,m} = 0, \forall i \geq 2,$ and the faction of jobs dropped is negligible, then we can obtain the following equilibrium
$$s_{1,m}^* = \lambda \frac{\prod_{i=1}^{m-1} p_i}{\mu_m}=\lambda v_m\quad \hbox{and}\quad s^*_{i,m}=0 \quad \forall i\geq 2.$$

We call $s^*$ {\em zero-waiting equilibrium} because it is a conjectured equilibrium by assuming zero-waiting. In this following analysis, we will not solve mean-field model \eqref{mfm:s1}-\eqref{mfm:sm} to check whether its equilibrium is close to $s^*$. Instead, we will directly prove $S_{1,m}$ concentrates around $s_{1,m}^*$ and zero waiting occurs at the steady-state with a high probability.  
\def\ThmmainA{\ref{Thm:equilibrium}}
\def\LemmainA{\ref{lem:equilibrium}}
\def\ThmmainB{\ref{Thm:main}}
\def\LemmainB{\ref{lem:main}}

\section{Iterative State-Space Peeling (ISSP)}\label{sec: issc}   
In this section, we illustrate the key idea  of ISSP, which will be applied to prove Theorem \ref{Thm:equilibrium}. 
Intuitively, the original stochastic system $S_{i,m}(t)$ and the steady-state $S_{i,m}$ would be close to the MFM $s_{i,m}(t)$ and  the zero-waiting equilibrium $s_{i,m}^*,$ respectively. 
However, due to ``non-monotonicity'' of the system, it is extremely challenging to justify this argument. To tackle the challenge, we develop a new technique, called ``Iterative State-Space Peeling'' (ISSP), which first identifies an iterative relation between the upper and lower bounds on the queue states, and then proves that the system lives in a regime concentrated around the ``fixed point' of the iterative bounds with a high probability. 

In particular, we focus on $S_{1,m}$ the number of busy servers with the job in service in phase $m$ because we hypothesize $S_{i,m} \to 0, \forall i \geq 2$ (in other words, the fraction of servers with any waiting jobs is negligible in a large system). 

Let $L_{1,m}(n)$ denote a high probability  lower bound on $S_{1,m}$ and $U_m(n)$ be a high probability upper bound on $\sum_{r=2}^{m} S_{1,r},$ established at the $n$th step of ISSP, i.e.
\begin{align}
    \mathbb P(S_{1,m} \geq L_{1,m}(n)) \quad \text{and}~~~~ \quad\mathbb P\left(\sum_{r=2}^{m} S_{1,r} \leq U_{m}(n)\right) \nonumber
\end{align}
are close to one. Our goal is to show that as $n$ increases, $L_{1,m}(n)$ and $U_m(n)$ approach the zero-waiting equilibrium $s_{1,m}^*$. Taking Coxian-$3$ distribution as an example, we need to show that as $n$ increases $L_{1,m}(n)\to s_{1,m}^*, \forall m,$ $U_2 \to s_{1,2}^*$ and $U_3 \to s_{1,2}^*+s_{1,3}^*.$

\subsection{ISSP: Iterative lower and upper bounds}
Our ``iterative state-space peeling'' (ISSP) is based on the following iterative relation between the upper and lower bounds on the queue states, denoted by $L_{1,m}$ and $U_m, \forall m \geq 2:$
\begin{align}
    L_{1,1}(n+1) &\approx \min\{s_{1,1}^*, 1 - U_M(n)\}\label{eq: iter head tail}\\
    L_{1,m}(n+1) &\approx \frac{v_m}{v_{m-1}} L_{1,m-1}(n+1) \label{eq: iter lower bound} \\
    U_m(n+1) &\approx 1- a_m - b_m L_{1,1}(n+1) + a_m U_{m-1}(n+1)\label{eq: iter upper bound} 
\end{align}
where the initial condition $L_{1,m}(0) = 0, \forall m,$ and $U_{m}(0) = 1, \forall m \geq 2,$ and ``$\approx$" is used because we ignore diminishing terms (e.g., $\frac{\log N}{\sqrt{N}}$) in the equations \eqref{eq: iter head tail}-\eqref{eq: iter upper bound} for explaining the intuition. From\eqref{eq: iter head tail}-\eqref{eq: iter upper bound}, we can obtain an recursive equation for $L_{1,1}:$
\begin{align}
L_{1,1}(n+1) &\approx \min\{s_{1,1}^*, v_1 +\xi (L_{1,1}(n) - v_1) \}, ~ s_{1,1}^* = \lambda v_1  ~\text{and}~ 0 < \xi < 1.
\end{align} Therefore, as $n\rightarrow \infty,$ $L_{1,m}(n) \to s_{1,m}^*$ and $U_{m}(n) \to \sum_{r=2}^{m} s_{1,r}^*.$

To provide the intuition behind \eqref{eq: iter head tail}-\eqref{eq: iter upper bound}, we consider the mean-field model under JSQ as an example and  focus on $s_{1,m}$ in \eqref{mfm:s1}-\eqref{mfm:sm} by ignoring $s_{i,m}, \forall i \geq 2,$ that is,
\begin{align}
\dot s_{1,1} &= \lambda \mathbb I(s_1 < 1) - \mu_1 s_{1,1}, \label{mfm:s11-JSQ}\\
\dot s_{1,m} &=  p_{m-1}\mu_{m-1}s_{1,m-1} - \mu_m s_{1,m}, ~\forall m \geq 2, \label{mfm:s1m-JSQ} 
\end{align}
where the equilibrium can be verified to be $s_{1,m}^* = \lambda v_m, \forall m.$ 

We next carefully analyze  \eqref{mfm:s11-JSQ}-\eqref{mfm:s1m-JSQ} to establish \eqref{eq: iter head tail}-\eqref{eq: iter upper bound}. To derive the lower and upper bounds on the equilibrium point of a dynamical system $x(t)$ (or $s_{1,m}(t)$), we use the following straightforward ideas.
\begin{itemize}
    \item If
    \begin{align}
    \dot x(t) > L - x(t), \label{eq:lb}
    \end{align}
    then $x(t) \geq L$ when $t$ is sufficiently large because otherwise $x(t)$ continues to increase.  
    \item If
    \begin{align}
    \dot x(t) < U - x(t), \label{eq:ub}
    \end{align}
    then $x(t) \leq U$ when $t$ is sufficiently large because otherwise, $x(t)$ continues to decrease.
\end{itemize}
In the following, we will explain \eqref{eq: iter head tail}-\eqref{eq: iter upper bound} based on the ideas above. The explanation is not a rigorous proof. The detailed proof will be presented later.
We will ignore iteration index $n$ occasionally when confusion does not arise.

\textbf{The intuition to obtain \eqref{eq: iter head tail}:}  
We start with the dynamic of $s_{1,1}$ in \eqref{mfm:s11-JSQ}, which is
\begin{align}
\dot s_{1,1} &= \lambda \mathbb I(s_1 < 1) - \mu_1 s_{1,1}, \nonumber
\end{align}
Given $\sum_{m=2}^M s_{1,m} < U_M,$ we have $$\dot s_{1,1} = \lambda - \mu_1 s_{1,1}$$ when $s_{1,1} < 1-U_M,$ which implies $\dot s_{1,1} > 0$ when $s_{1,1} < \min\left\{1-U_M, s_{1,1}^*\right\}.$ Therefore, at the equilibrium point, $s_{1,1} \geq L_{1,1} \triangleq \min\{1-U_M, s_{1,1}^*\}$ because otherwise, $s_{1,1}$ will continue to increase because $\dot s_{1,1} > 0.$ 

\textbf{The intuition to obtain \eqref{eq: iter lower bound}:} Consider the dynamic of $s_{1,m}, \forall m \geq 2$ in \eqref{mfm:s1m-JSQ}: 
$$\dot s_{1,m} = p_{m-1}\mu_{m-1}s_{1,m-1} - \mu_m s_{1,m}, \forall m \geq 2.$$
Given $s_{1,m-1} \geq L_{1,m-1},$ we have $$\dot s_{1,m} \geq p_{m-1}\mu_{m-1}L_{1,m-1} - \mu_m s_{1,m}, \forall m \geq 2,$$
which implies at the equilibrium point, 
$$s_{1,m} \geq  L_{1,m} \triangleq \frac{p_{m-1}\mu_{m-1}}{\mu_m }L_{1,m-1} = \frac{v_m}{v_{m-1}} L_{1,m-1}$$ because otherwise, $s_{1,m}$ will continue to increase because $\dot s_{1,m} > 0.$

Note $\frac{v_m}{v_{m-1}} = \frac{s_{1,m}^*}{s_{1,m-1}^*},$ so we have $ \frac{L_{1,m}}{L_{1,m-1}}=\frac{s_{1,m}^*}{s_{1,m-1}^*},$ which means the ratio of the lower bounds is the same as that of the corresponding equilibrium points.      

\textbf{The intuition to obtain \eqref{eq: iter upper bound}:} We focus on the dynamic of $\sum_{r=2}^m s_{1,r}$ that 
\begin{align}
\sum_{r=2}^m \dot s_{1,r} &= p_1\mu_1s_{1,1} - \sum_{r=2}^{m-1} (1-p_r)\mu_r s_{1,r} - \mu_m s_{1,m} \nonumber \\
&\leq p_1\mu_1\left(1-\sum_{r=2}^{M} s_{1,r}\right) - \sum_{r=2}^{m-1} (1-p_r)\mu_r s_{1,r} - \mu_m s_{1,m} \nonumber \\
&= p_1\mu_1\left(1-\sum_{r=m+1}^{M} s_{1,r}\right) - \sum_{r=2}^{m-1} (1-p_r)\mu_r s_{1,r} - (p_1\mu_1 + \mu_m) \sum_{r=2}^{m}s_{1,r}  + \mu_m\sum_{r=2}^{m-1} s_{1,r}. \nonumber
\end{align}
Given $s_{1,m} \geq L_{1,m}, \forall m$ and $\sum_{r=2}^{m-1} s_{1,r} \leq U_{m-1},$ we have
\begin{align}
\sum_{r=2}^m \dot s_{1,r} \leq p_1\mu_1\left(1-\sum_{r=m+1}^{M} L_{1,r}\right) - \sum_{r=2}^{m-1} (1-p_r)\mu_r L_{1,r}- (p_1\mu_1 + \mu_m) \sum_{r=2}^{m}s_{1,r} + \mu_mU_{m-1}, \nonumber
\end{align}
which implies at the equilibrium point,
\begin{align}
\sum_{r=2}^{m}s_{1,r} \leq U_m  \triangleq \frac{p_1\mu_1\left(1-\sum_{r=m+1}^{M} L_{1,r}\right) - \sum_{r=2}^{m-1} (1-p_r)\mu_r L_{1,r} + \mu_mU_{m-1}}{p_1\mu_1+\mu_m} \nonumber 
\end{align}
because otherwise, $\sum_{r=2}^m  s_{1,r}$ will continue to decrease because $\sum_{r=2}^m \dot s_{1,r} < 0.$ 

By invoking $L_{1,m} = \frac{v_m}{v_{m-1}}L_{1,m-1},$ we have 
\begin{align}
\sum_{r=m+1}^{M} L_{1,r} &= \sum_{r=m+1}^M  \frac{v_r}{v_1}L_{1,1}, \nonumber\\
\sum_{r=2}^{m-1} (1-p_r)\mu_r L_{1,r} &= p_1\mu_1 L_{1,1} - \mu_m L_{1,m} = p_1\mu_1 L_{1,1} - \frac{\mu_m v_m}{v_1} L_{1,1}. \nonumber   
\end{align}
Therefore, we have
\begin{align}
U_m &= \frac{p_1\mu_1}{p_1\mu_1+\mu_m} - \frac{p_1\mu_1}{p_1\mu_1+\mu_m} \left(1+\sum_{r=m+1}^{M} \frac{v_m}{v_1}\right) L_{1,1} + \frac{\mu_m}{p_1\mu_1+\mu_m}\frac{v_m}{v_1} L_{1,1} + \frac{\mu_m}{p_1\mu_1+\mu_m}U_{m-1} \nonumber \\
&= \frac{p_1\mu_1}{p_1\mu_1+\mu_m} - \left(\frac{p_1\mu_1}{p_1\mu_1+\mu_m} \left(1+\sum_{r=m+1}^{M} \frac{v_m}{v_1}\right) - \frac{\mu_m}{p_1\mu_1+\mu_m}\frac{v_m}{v_1}\right)L_{1,1} + \frac{\mu_m}{p_1\mu_1+\mu_m}U_{m-1} \nonumber\\
&= 1-a_m -b_m L_{1,1} + a_m U_{m-1}. \nonumber
\end{align}
where the last equality holds by the definitions of $a_m$ and $b_m.$

\subsection{ISSP: An illustrative example}
To demonstrate ISSP, we consider JSQ  with Erlang-$3$ distribution and no buffer, i.e., Erlang-$3$ with $b = 1.$ 

\textbf{The mean-field model under JSQ with Erlang-$3$ and $b = 1.$}

With the Erlang-$3$ service time distribution, we have $$p_1=p_2=p_3=1\quad\hbox{and}\quad\mu_1=\mu_2=\mu_3=3.$$
So the MFM in this case is
\begin{align}
\dot s_{1,1} &= \lambda \mathbb I_{\{s_1 < 1\}} - 3 s_{1,1} \nonumber \\ 
\dot s_{1,2} &=  3 s_{1,1} - 3 s_{1,2}\nonumber \\
\dot s_{1,3} &= 3 s_{1,2} - 3 s_{1,3}  \nonumber \\
s_{i,m} &\equiv 0, \forall i \geq 2, \forall m. \nonumber
\end{align} 

\textbf{ISSP for JSQ with Erlang-$3$ and $b = 1.$} 
The values of the key parameters in iterative equations in this case are $a_1 = a_2 = a_3 = \frac{1}{2},$ $b_1 = 1, b_2 = \frac{1}{2}, b_3 = 0,$ and  $v_1=v_2=v_3=\frac{1}{3}.$
The corresponding iterative equations are 
\begin{align}
    L_{1,1}(n+1) \approx \min &\left\{\frac{\lambda}{3}, 1  - U_3(n) \right\}  \nonumber \\
    L_{1,2}(n+1) \approx L_{1,1}(n+1),& \quad \hbox{ and }~ L_{1,3}(n+1) \approx L_{1,2}(n+1) \nonumber\\
    U_2(n+1) \approx \frac{1}{2} - \frac{1}{2}L_{1,1}(n+1),& \quad \hbox{ and } ~ U_3(n+1) \approx \frac{1}{2} + \frac{1}{2}U_2(n+1) \nonumber
\end{align}
The iterative relation in terms of $L_{1,1}(n)$ is
\begin{align}
    L_{1,1}(n+1) =  \min \left\{\frac{\lambda}{3}, \frac{1}{4}+\frac{1}{4}L_{1,1}(n)\right\}\nonumber
\end{align}
which implies that $L_{1,1}(n) \to \frac{\lambda}{3}$ as $n\rightarrow\infty,$ and 
\begin{align*}
L_{1,1}(n) \to \frac{\lambda}{3}, L_{1,2}(n) \to \frac{\lambda}{3}, L_{1,3}(n) \to \frac{\lambda}{3},
U_2(n) \to \frac{\lambda}{3}, U_3(n) \to \frac{2\lambda}{3}. 
\end{align*}
The iterative procedure has been visualized in Figure \ref{fig:erlang 3}. In each iteration $n,$ we first establish $L_{1,1}(n+1)$ in light red based on $U_3(n)$ from the last iteration; then obtain $L_{1,1}(n+1)\approx L_{1,2}(n+1)\approx L_{1,3}(n+1)$ in light (blue, green); and finally refine $U_2(n+1)$ and $U_3(n+1)$ in light purple, which in turn will improve $L_{1,1}$ in the next iteration. 
\begin{figure}
  \centering
  \includegraphics[width=5.3in]{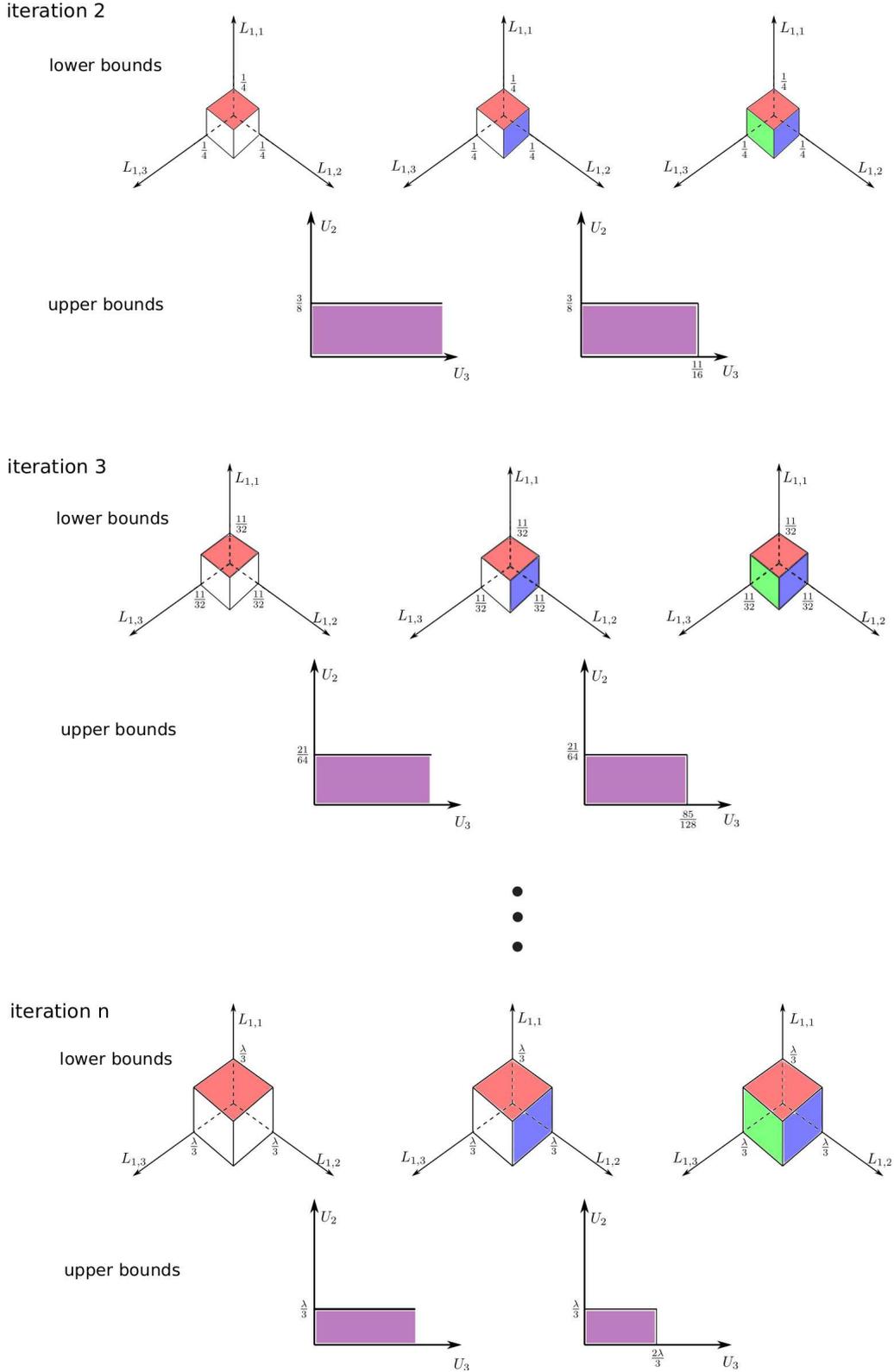}
  \caption{~Illustrations of ISSP  under Erlang-$3$ service time distribution. The lower bounds of $s_{1,m}, \forall m,$ keep increasing and the upper bounds of $s_{1,2}$ and $s_{1,2}+s_{1,3}$ keep decreasing until reaching the equilibriums: the initial values (at iteration $1$) are $L_{1,m}(1)=0,\forall m, U_2(1)=1,U_3(1)=1;$ the lower bounds $L_{1,m}$ increase as $0 \to \frac{1}{4}\to \frac{11}{32} \to \cdots \to \frac{\lambda}{3}$; the upper bound $U_2$ of $s_{1,2}$ decreases as $1 \to \frac{3}{8} \to \frac{21}{64} \to \cdots \to \frac{\lambda}{3};$ the upper bound $U_3$ of $s_{1,2}+s_{1,3}$ decreases as $1 \to \frac{11}{16} \to \frac{85}{128} \to \cdots \to \frac{2\lambda}{3}.$ 
  \label{fig:erlang 3}}
\end{figure}

In the following sections, we formalize the ISSP, which is then combined with Stein's method to prove Theorems \ref{Thm:equilibrium} and \ref{Thm:main}. A roadmap can be found in Figure \ref{fig:roadmap} that demonstrates the relationship of the key lemmas and theorems.
\begin{figure}[H]
  \centering
  \includegraphics[width=5.3in]{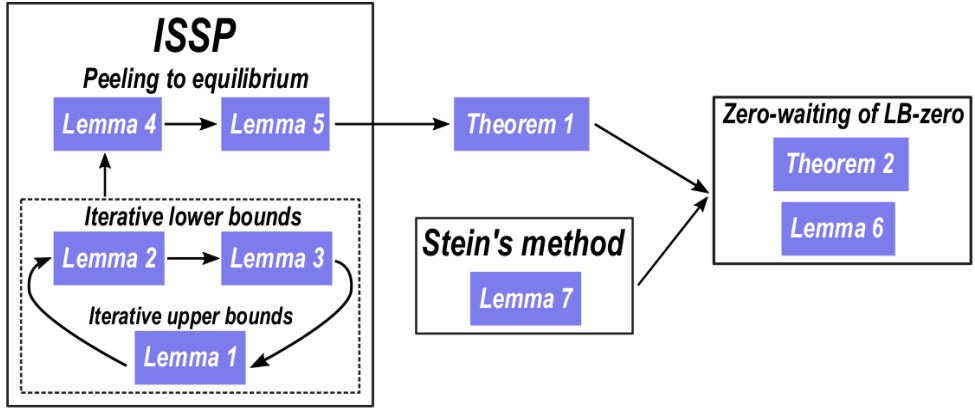}
  \caption{~A Roadmap for Proving Theorems \ref{Thm:equilibrium} and \ref{Thm:main}: Lemmas \ref{lem: iter head tail}, \ref{lem: iter prob lower bound}, and \ref{lem: iter prob upper bound}  establish the iterative upper and lower bounds, which are used to establish Lemma \ref{lem: L11 lift}, \ref{lem: S11 lower bound} and Theorem \ref{Thm:equilibrium}. Together with Stein's method in Lemma \ref{lem:gen diff},  zero-waiting is established in  Lemma \ref{lem:main} and Theorem \ref{Thm:main}.}
  \label{fig:roadmap}
\end{figure}

\section{Proof of Theorem \ThmmainA~ based on ISSP}\label{sec: prove thm 1}
In this section, we present the formal statements of the iterative equations \eqref{eq: iter head tail}-\eqref{eq: iter upper bound} as individual lemmas, which are used to prove Theorem \ThmmainA. 

Recall that the positive constant $C = \sqrt{\frac{2\bar v^2\log (1/\xi)}{3M + (3M+4)\log (1/\xi)}},$ whose value is chosen to control the tail probability as we will see it later. Recall $a_m = \frac{\mu_m}{p_1\mu_1 + \mu_m},$ $b_m = (1-a_m)\left(1 + \sum_{r=m+1}^M \frac{v_{r}}{v_{1}}\right) - \frac{a_m v_m}{v_1},$ and $\bar v = \min_m v_m.$ Recall $c_m = 5(1-a_m)\sum_{r=m+1}^M (r-1)v_r + 5a_m \sum_{r=2}^{m-1} \frac{\mu_r v_r}{\mu_m} + 5(m-2)a_m v_m + 5 - a_m,$ $C_M = \sum_{m=2}^{M} c_m\prod_{j=m+1}^M a_j,$ and $\Delta = \frac{\log N}{\sqrt{N}}.$ We further define 
\begin{align*}
\epsilon_{1}(n+1)=& e^{-\frac{\log^2 N}{C^2}} +\left(\frac{C}{\Delta}+1\right)\sigma_{M}(n)\\ \epsilon_{m}(n+1)=& e^{-\frac{v_m^2\log^2 N}{C^2}} + \left( \frac{C}{v_m\Delta}+1\right)\epsilon_{m-1}(n+1)\\ \sigma_{m}(n+1)=&e^{-\frac{\log^2 N}{C^2}} +\left(\frac{C}{\Delta}+1\right) \left(\sigma_{m-1}(n+1) + \sum_{m=1}^M \epsilon_m(n+1)\right)
\end{align*}
with initial values $\epsilon_m(0) = \sigma_m(0) =0, \forall m.$ These are the constants used  in the tail probabilities in the following lemmas. We only state the lemmas and their proofs can be found in Appendix \ref{App:IterSSC}.

\subsection{A lower bound on $L_{1,1}(n+1)$ given $\sum_{m=2}^{M} S_{1,m} \leq U_{M}(n).$}
The following lemma is the rigorous statement of \eqref{eq: iter head tail}. 
\begin{restatable}{lemma}{iterheadtail}\label{lem: iter head tail}
Given $$\mathbb P\left(\sum_{m=2}^{M} S_{1,m} > U_{M}(n)\right) \leq \sigma_{M}(n),$$ and defining
\begin{align*}
L_{1,1}(n+1) =& \min\left\{s_{1,1}^* - \frac{6\Delta}{C}, 1 - U_{M}(n) - \frac{1-\xi}{2\mu_1N^\alpha} -\frac{6\Delta}{C} \right\}, 
\end{align*}
we have $$\mathbb P\left(S_{1,1} < L_{1,1}(n+1)\right) \leq \epsilon_{1}(n+1).$$ \qed
\end{restatable}

\subsection{A lower bound on $S_{1,m}$ given $S_{1,m-1} \geq L_{1, m-1}.$}
The following lemma is the rigorous statement of  \eqref{eq: iter lower bound}.

\begin{restatable}{lemma}{iterproblowerbound}\label{lem: iter prob lower bound}
Consider $m\geq 2.$ Given $$\mathbb P\left(S_{1,m-1} < L_{1,m-1}(n+1)\right)\leq \epsilon_{m-1}(n+1),$$
and defining $$L_{1,m}(n+1) = \frac{v_m}{v_{m-1}}L_{1,m-1}(n+1)-\frac{5v_m}{C}\Delta,$$
we have $$\mathbb P\left(S_{1,m} < L_{1,m}(n+1)\right) \leq \epsilon_{m}(n+1).$$ \qed
\end{restatable}

\subsection{An upper bound on $\sum_{r=2}^{m} S_{1,r}$ given $\sum_{r=2}^{m-1} S_{1,r} \leq U_{m-1}$ and $S_{1,m-1} \geq L_{1,m-1}.$}
The following lemma is the rigorous statement of \eqref{eq: iter upper bound}.
\begin{restatable}{lemma}{iterprobupperbound}\label{lem: iter prob upper bound}
Consider $m\geq 2.$ Given 
\begin{align}
&\mathbb P\left(\sum_{r=2}^{m-1} S_{1,r} \geq U_{m-1}(n+1)\right) \leq \sigma_{m-1}(n+1) \nonumber\\
&\mathbb P\left(S_{1,r} < L_{1,r}(n+1)\right)\leq \epsilon_{r}(n+1) \quad 1\leq r\leq M\nonumber
\end{align}
and defining $$U_{m}(n+1) = 1 - a_m - b_m L_{1,1}(n+1) + a_mU_{m-1}(n+1) + \frac{c_m}{C}\Delta,$$ 
we have $$\mathbb P\left(\sum_{r=2}^m S_{1,r} \geq U_{m}(n+1)\right) \leq \sigma_{m}(n+1).$$ \qed
\end{restatable}

\subsection{Convergence of $L_{1,1}(n).$}
Based on Lemmas \ref{lem: iter head tail}, \ref{lem: iter prob lower bound}, and \ref{lem: iter prob upper bound}, we will show $\{L_{1,1}(n)\}_n$ is an increasing sequence and approaches $s_{1,1}^*.$  
\begin{restatable}{lemma}{loneonelift}\label{lem: L11 lift}
Recall that $\xi=\sum_{m=2}^{M} b_m \prod_{j=m+1}^M a_j.$ Given $\mathbb P(S_{1,1} < L_{1,1}(n)) \leq \epsilon_1(n)$ and $L_{1,1}(n) \leq s_{1,1}^* - \frac{6\Delta}{C},$ we have $$\mathbb P(S_{1,1} < L_{1,1}(n+1)) \leq \epsilon_1(n+1),$$
where  $$L_{1,1}(n+1) = \frac{1}{\mu_1}-\frac{C_M\Delta}{C(1-\xi)} - \frac{1}{2\mu_1 N^\alpha} +  \xi\left(L_{1,1}(n) - \frac{1}{\mu_1}+\frac{C_M\Delta}{C(1-\xi)} + \frac{1}{2\mu_1 N^\alpha}\right)$$ and $\epsilon_1(n+1) = \epsilon_1(n) (M^2+2) (\frac{2C}{\bar v\Delta}+1)^{2M}.$ Furthermore, $L_{1,1}(n+1) > L_{1,1}(n)$ holds when $L_{1,1}(n) \leq s_{1,1}^* - \frac{6\Delta}{C}.$ \qed
\end{restatable}

\subsection{Proving Theorem \ThmmainA}
Based on the monotonicity of $L_{1,1}(n)$,  we can apply Lemma \ref{lem: L11 lift} a sufficient number of times so that $L_{1,1}(n)$ is close to $s_{1,1}^* - \frac{6\Delta}{C},$ which is formalized in the following lemma.
\begin{lemma}\label{lem: S11 lower bound}
$$\mathbb P\left(S_{1,1} < s_{1,1}^* - \frac{ 6\Delta}{C}\right) \leq \left(\frac{1}{\sqrt{N}}\right)^{M+6}.$$
\end{lemma}
\begin{proof}
To prove this lemma, we apply Lemma \ref{lem: L11 lift} $n$ times iteratively with $n=\left \lceil{\frac{\log N}{2\log(1/\xi)}}\right \rceil 
$ such that $\xi^{n} \leq \Delta.$
We obtain 
\begin{align*}
\mathbb P\left(S_{1,1} < \frac{1}{\mu_1} -\frac{C_M\Delta}{(1-\xi)C}-\frac{1}{2\mu_1N^\alpha} -\frac{\Delta}{\mu_1}\right) \leq& \epsilon_1(0) (M^2+2)^n \left(\frac{2C}{\bar v\Delta}+1\right)^{2Mn}\\
\leq& \epsilon_1(0) (M^2+2)^{\frac{\log N}{\log (1/\xi)}+1} N^{\frac{2M\log N}{\log (1/\xi)}+2M}\\
\leq & e^{-\frac{\bar v^2 \log^2 N}{C^2}}(M^2+2)^{\frac{\log N}{\log (1/\xi)} + 1} N^{\frac{2M\log N}{\log (1/\xi)} + 2M}\\
=& N^{-\frac{\bar v^2\log N}{C^2}+\frac{2M\log N}{\log (1/\xi)}+2M}(M^2+2)^{\frac{\log N}{\log (1/\xi)}+1}\\
\leq& N^{-\frac{\bar v^2\log N}{C^2}+\frac{3M\log N}{\log (1/\xi)}+2M+1},
\end{align*}
where the third inequality holds because $\epsilon_1(0) \leq e^{-\frac{\bar v^2 \log^2 N}{C^2}},$ and the last inequality holds because $N \geq M^2+2.$

Recalling $C = \sqrt{\frac{2\bar v^2\log (1/\xi)}{3M + (3M+4)\log (1/\xi)}}$ and noting $N \geq 1/\xi,$ we have 
\begin{align*}
&\mathbb P\left(S_{1,1} \leq \frac{1}{\mu_1} -\frac{C_M\Delta}{(1-\xi)C} -\frac{\Delta}{\mu_1}-\frac{1}{2\mu_1N^\alpha}\right)\\
=&\mathbb P\left(S_{1,1} \leq \frac{\lambda}{\mu_1} -\frac{C_M\Delta}{(1-\xi)C} -\frac{\Delta}{\mu_1}+\frac{1}{2\mu_1N^\alpha}\right) \leq \left(\frac{1}{\sqrt{N}}\right)^{M+8},    
\end{align*}
which implies
$$\mathbb P\left(S_{1,1} < s_{1,m}^* - \frac{6\Delta}{C}\right) \leq \left(\frac{1}{\sqrt{N}}\right)^{M+8}$$
because $N^{0.5-\alpha} \geq \frac{2\mu_1C_M}{C(1-\xi)}\log N.$
\end{proof}

The result above established that $$S_{1,1} \geq s_{1,1}^* - \frac{6\Delta}{C} = \lambda v_1 - \frac{6\Delta}{C},$$ with probability $1-\left(\frac{1}{\sqrt{N}}\right)^{M+8}.$ 

Combining with Lemma \ref{lem: iter prob lower bound}, we next prove that $S_{1,m} \geq s_{1,m}^* - \Omega(\Delta), \forall m \geq 2$ holds with a high probability.
Applying Lemma \ref{lem: iter prob lower bound} iteratively for $S_{1,m}, \forall m \geq 2,$ we have 
\begin{align*}
L_{1,m}(n) =& \frac{v_m}{v_1}L_{1,1}(n) - \frac{5(m-1)v_m\Delta}{C}\\
=&\frac{v_m}{v_1} \left(\lambda v_1 - \frac{6\Delta}{C}\right) - \frac{5(m-1)v_m\Delta}{C} \\
=&\lambda v_m - \frac{6\mu_1 v_m +  5(m-1)v_m\Delta}{C},
\end{align*}
and $S_{1,m} \geq L_{1,m}(n)$ holds with the probability $\epsilon_m(n).$ Note $\epsilon_m(n) \leq \epsilon_M(n), \forall m$ and $\epsilon_M(n)$ is bounded as follows
\begin{align*}
\epsilon_M(n) \leq M\epsilon_1(n)\left(\frac{C}{\bar v \Delta}+1\right)^{M-1}
=M\left(\frac{1}{\sqrt{N}}\right)^{M+8} \left(\frac{C}{\bar v \Delta}+1\right)^{M-1}
\leq \frac{1}{N^3}.
\end{align*}
Therefore, we have proved the lower bound in the theorem.

Define the event $\mathcal K = \{S_{1,m} \geq s_{1,m}^* + \frac{\theta_m\log N}{\sqrt{N}}, \forall m\}.$ We have $\mathbb P({\mathcal K}^c) \leq \frac{M}{N^3}$ according to the union bound. We now establish the upper bound in Theorem \ThmmainA~ as follows
\begin{align*}
    1 &= \mathbb P\left(S_{1,m} \leq 1 - \sum_{r\neq m} S_{1,r} \right)\\
      &\leq \mathbb P\left(S_{1,m} \leq 1 - \sum_{r\neq m} S_{1,r} ~|~ \mathcal K\right) +  \mathbb P({\mathcal K}^c)\\
      &\leq \mathbb P\left(S_{1,m} \leq 1 - \sum_{r\neq m} \left(s_{1,r}^* + \frac{\theta_r\log N}{\sqrt{N}}\right)\right) +  \frac{M}{N^3}.
\end{align*}
The proof is completed because  $1-\sum_{r\neq m}^M s_{1,r}^* = s_{1,m}^* +  \frac{1}{N^\alpha}.$ 

\section{Proof of Theorem \ThmmainB}\label{sec: prove thm 2} 
Theorem \ThmmainA~ shows that $S_{1,m}$ is ``close'' to $s_{1,m}^*$ with a high probability. However, it is not clear whether the average total queue length or the waiting probability is small under an LB-zero policy. To establish Theorem \ref{Thm:main}, we first prove an important lemma on the upper bound of the average total queue, which is used to establish the waiting probability in Theorem \ref{Thm:main}.   The proof of this lemma can be found in Appendix \ref{app:lem main}. 
\begin{lemma}\label{lem:main}
Define $w_m = (1-p_m)\mu_m,$ $w_u=\max_m w_m,$ $w_l=\min_m w_m,$ $\mu_{\max}=\max_m \mu_m,$ $\zeta = \frac{4w_ub}{w_l}\left((\frac{1}{w_l}-\frac{1}{w_u})\sum_m \theta_m w_m + \frac{1}{w_l}+6\right)$ and $k = \frac{\sum_{m}\theta_m w_m}{w_u}+(1+\frac{w_l}{4w_ub})\zeta - \sum_m \theta_m.$  Under a load balancing policy in LB-zero, the following bound holds 
\begin{align}
\mathbb E\left[\max\left\{\sum_{i=1}^{b} S_i- \lambda -\frac{k\log N}{\sqrt{N}},0\right\}\right] \leq \frac{5\mu_{\max}+2}{\sqrt{N}\log N}, \nonumber
\end{align}
when $N$ satisfies
\begin{align}
 \min\left\{2k\mu_1,\sum_{m=1}^M\theta_m,\frac{C(1-\xi)}{2\mu_1 C_M}\right\} N^{0.5-\alpha} \geq \log N \geq \max\left\{\log\left(\frac{1}{\xi}\right), \frac{2\mu_1}{1-\xi}, \frac{4b}{w_l \zeta}, \frac{C}{\mu_1}\right\}. \nonumber
\end{align}
\hfill{$\square$}
\end{lemma}
Lemma \ref{lem:main} establishes an upper bound on the average total queue length. Recall Theorem \ThmmainA~ indicates the service rate under an LB-zero policy in $\Pi$ is close to arrival rate $\lambda N$ at steady state because $\sum_{m=1}^M (1-p_m)S_{1,m} \approx \sum_{m=1}^M (1-p_m)s_{1,m}^* =\lambda.$ Therefore, it is reasonable to couple the distributed load balancing system with a simple centralized server system with a similar arrival rate and service rate. This coupling will be done via Stein's method \cite{BraDaiFen_16,Yin_16,Bra_21}. Stein's method allows us to understand the key performance metrics (e.g., average queue length) of a complicated load balancing system from the performance of a simple fluid system (to be introduced  in the next section). Formally, we study the generator difference between the distributed load balancing system and a simple centralized system within a small state space identified by  ISSP. This idea of coupling a simple (and almost trivial) fluid model, without ISSP, has been also used in \cite{LiuYin_20,LiuYin_21,LiuGonYin_21}. We will introduce it next so the paper is self-contained.

\subsection{Generator Coupling with a Single Server System}

Denote $\Delta = \frac{\log N}{\sqrt{N}}.$ We consider a single server queue with arrival rate $\lambda$ and service rate $\lambda+\Delta.$ The fluid model with respect to the queue length $x$ is  
\begin{align}
    \dot{x}=\frac{dx}{dt}=-\Delta. \label{eq:simple fluid}
\end{align} Let function $g(x)$ be the solution of the following Stein's equation or Poisson equation \cite{Yin_16} for the fluid system above and distance function $h(x)$ such that:
\begin{align}
\frac{dg(x)}{dt} = g'(x) (-\Delta) = h(x),  \forall x, \label{Gen:L}
\end{align}
where {$g'(x) = \frac{d g(x)}{d x}$}. Because \eqref{Gen:L} has a very simple form, both $g'$ and $g''$ can be easily solved which is different from  other applications of Stein's method, e.g., \cite{BraDaiFen_16}, where establishing the gradient bounds is a key difficulty.

To analyze the total queue length at steady-state under an LB-zero policy in $\Pi,$ we choose a truncated distance function:
$$h\left(\sum_{i=1}^bS_i\right) = \max\left\{\sum_{i=1}^bS_i-\eta,0\right\}, ~~ \eta =\lambda +k\Delta.$$ 
The distance $h\left(\sum_{i=1}^bS_i\right)$ can be viewed as a proxy to measure the total queue length ($N\sum_{i=1}^b S_i$) at steady state. 

To couple the one-dimensional fluid system in \eqref{eq:simple fluid} with the $b\times M$-dimensional stochastic system, we define
\begin{equation}
f(s)=g\left(\sum_{i=1}^b s_{i}\right)=g\left(\sum_{i=1}^b\sum_{m=1}^M s_{i,m}\right).\label{eq:steinsolution}
\end{equation} 

Note $f(s)$ is bounded for $s\in\mathcal{S}^{(N)},$ we impose the generator of stochastic system $G$ on function $f$ and have the basic adjoint relationship for the stationary distribution $S$ such that
\begin{equation}
\mathbb E[Gf(S)]=\mathbb E\left[Gg\left(\sum_{i=1}^b\sum_{m=1}^M S_{i,m}\right)\right]=0.\label{eq:sse}
        \end{equation}

Combining \eqref{Gen:L} and \eqref{eq:sse}, we connect the performance metric $h(\cdot)$ with the generator difference between the simple single-server system and $G$ as follows
\begin{align}
\mathbb E\left[h\left(\sum_{i=1}^b\sum_{m=1}^M S_{i,m}\right)\right] 
= \mathbb E\left[ g'\left(\sum_{i=1}^b\sum_{m=1}^M S_{i,m}\right) (-\Delta) - Gg\left(\sum_{i=1}^b\sum_{m=1}^M S_{i,m}\right)\right].\label{eq:gencou}
\end{align}
In the following lemma, we provide an upper bound on \eqref{eq:gencou} which includes two terms: the first term is from the gradient bounds and the second term is from ISSP. The proof of this lemma can be found in Appendix \ref{app:gen-diff}.
\begin{lemma}\label{lem:gen diff}
Define the regions $\mathcal T = \{ x ~|~ x > \eta+\frac{1}{N}\}$ and denote the normalized service rate $D_1=\sum_{m=1}^M(1-p_m)\mu_m S_{1,m},$ we have
\begin{align}
&\mathbb E\left[h\left(\sum_{i=1}^b  S_{i}\right)\right]
\leq J_1 +  \frac{5\mu_{\max}+\lambda}{\sqrt{N}\log N}, \label{eq:gen diff}
\end{align}
with
\begin{align}
J_1 =& \mathbb E\left[g'\left(\sum_{i=1}^b S_{i}\right)\left(\lambda A_b(S) - \lambda-\Delta+D_1\right)\mathbb{I}_{\sum_{i=1}^b S_{i} \in \mathcal T}\right]. \label{G-expansion-SSC}
\end{align}
\qed
\end{lemma}
To establish Lemma \ref{lem:main} and Theorem \ref{Thm:main} based on the lemmas above, we need to provide the upper bounds on \eqref{G-expansion-SSC}, which is related to the difference between the normalized arrival rate and the normalized service rate, which will be bounded based on the ISSP result that shows the the normalized service rate at the steady-state is close to the zero-waiting equilibirium value.

\subsection{State Space Peeling on $\sum_{m=1}^M(1-p_m)S_{1,m}$}\label{sec:SSC}

In this subsection, we analyze $J_1$ in \eqref{G-expansion-SSC}:
\begin{align}
J_1=& \mathbb E\left[\frac{1}{\Delta} h\left(\sum_{i=1}^bS_{i}\right) \left(-\lambda A_b(S) + \lambda + \Delta-D_1\right)\mathbb{I}_{\sum_{i=1}^b S_{i} > \eta + \frac{1}{N}}\right]\nonumber\\
\leq& \mathbb E\left[\frac{1}{\Delta} h\left(\sum_{i=1}^b S_{i}\right) \left(\lambda + \Delta-D_1\right)\mathbb{I}_{\sum_{i=1}^b S_{i} > \eta + \frac{1}{N}}\right]\label{SSC0},
\end{align} where the first equality is due to the definition of $g'$ in Stein's equation \eqref{Gen:L}, and the inequality holds because $\frac{1}{\Delta} h\left(\sum_{i=1}^b S_{i}\right) \mathbb{I}_{\sum_{i=1}^b S_{i} > \eta + \frac{1}{N}}\geq 0.$
We focus on 
\begin{equation}
\left(\lambda + \Delta-\sum_{m=1}^M(1-p_m)\mu_m s_{1,m}\right)\mathbb{I}_{\sum_{i=1}^b s_{i} > \eta + \frac{1}{N}},\label{eq:sscterm}
\end{equation}
where we recall $\eta = \lambda +k\Delta$ and $d_1=\sum_{m=1}^M(1-p_m)\mu_m s_{1,m}$ is the total service rate when the system is in the state $s.$ Though we have established $S_{1,m} \approx s_{1,m}^*$ in Theorem \ThmmainA, it is not sufficient for showing \eqref{SSC0} is small enough because $\lambda + \Delta - D_1$ may be larger than $\Delta,$ which will make \eqref{SSC0} very large. In fact, we need one more state space peeling to show \eqref{SSC0} is non-positive.

We define two regions $\mathcal S_{ssp_1}$ and $\mathcal S_{ssp_2}$ 
\begin{align*}
\mathcal S_{ssp_1} =& \left\{s ~|~ s_1 \geq \lambda + \left(k-\zeta-6\right)\Delta, s_{1,m} \geq s_{1,m}^* - \theta_m\Delta\right\}, \\
\mathcal S_{ssp_2} =& \left\{s ~|~ \sum_{i=1}^{b} s_i \leq \lambda+k\Delta \right\},
\end{align*}
where $\mathcal S_{ssp_1}$ is the region with sufficient many busy servers and $\mathcal S_{ssp_2}$ is the region with bounded total queue length. We further define a region $$\mathcal S_{ssp} = \mathcal S_{ssp_1} \bigcup S_{ssp_2},$$ and consider two cases: $s\in \mathcal S_{ssp} $ and $s\not\in \mathcal S_{ssp},$
\begin{itemize}
    \item {\bf Case 1:}
    In Lemma \ref{SSC:s1 s2 negative} in the appendix, we show any $s\in \mathcal S_{ssp_1}$ satisfies $$\sum_{m=1}^M(1-p_m)\mu_m s_{1,m} \geq \lambda + \Delta.$$ It implies  $\left(\lambda + \Delta-\sum_{m=1}^M(1-p_m)\mu_m s_{1,m}\right)\mathbb{I}_{\sum_{i=1}^b s_{i} > \eta + \frac{1}{N}} \leq 0$ for any $s\in \mathcal S_{ssp_1}.$ For any $s\in \mathcal S_{ssp_2},$ we have $$\mathbb{I}_{\sum_{i=1}^b s_{i} > \eta + \frac{1}{N}}=0.$$ It implies $\left(\lambda + \Delta-\sum_{m=1}^M(1-p_m)\mu_m s_{1,m}\right)\mathbb{I}_{\sum_{i=1}^b s_{i} > \eta + \frac{1}{N}} =0$ for any $s\in \mathcal S_{ssp_2}.$ 
    
    \item {\bf Case 2:} In Lemma \ref{SSC:out} in the appendix, we show that $$\mathbb P\left( S \notin \mathcal S_{ssp}\right) \leq \frac{2}{N^{2}}$$ by using an ISSP approach on $S_1$ and $\sum_{i=2}^b S_i.$ 
\end{itemize}

\subsection{Proving Lemma \ref{lem:main}}\label{app:lem main}
Based on the two cases above, we split \eqref{SSC0} into two regions $s \in S_{ssp}$ and $s \notin S_{ssp}$  and obtain 
\begin{align}
\eqref{SSC0}=&\mathbb E\left[\frac{1}{\Delta}\left(\sum_{i=1}^bS_i-\eta\right)\left(\lambda+\Delta-D_1\right)\mathbb{I}_{S \in \mathcal S_{ssp}}\mathbb{I}_{\sum_{i=1}^bS_i> \eta +\frac{1}{N}}\right]\nonumber\\
&+\mathbb E\left[\frac{1}{\Delta}\left(\sum_{i=1}^bS_i-\eta\right)\left(\lambda+\Delta-D_1\right)\mathbb{I}_{S \notin \mathcal S_{ssp}}\mathbb{I}_{\sum_{i=1}^bS_i> \eta +\frac{1}{N}}\right] \nonumber\\
\leq&\mathbb E\left[\frac{1}{\Delta}\left(\sum_{i=1}^bS_i-\eta\right)\left(\lambda+\Delta-D_1\right)\mathbb{I}_{S \notin \mathcal S_{ssp}}\mathbb{I}_{\sum_{i=1}^bS_i> \eta +\frac{1}{N}}\right] \nonumber\\
\leq& \frac{2b}{N^{1.5}\log N},  \label{SSC-bound}
\end{align}
where the first inequality holds because of Lemma \ref{SSC:s1 s2 negative} and the second inequality holds because the average total number of jobs per server is at most $b$ and $\left(\lambda+\Delta-D_1\right)\mathbb{I}_{S \notin \mathcal S_{ssp}}\mathbb{I}_{\sum_{i=1}^bS_i> \eta +\frac{1}{N}} < 1.$

By combining \eqref{eq:gen diff} and \eqref{SSC-bound}, we can now establish Lemma \ref{lem:main} in the following
\begin{align}
&\mathbb E\left[\max\left\{\sum_{i=1}^{b} S_i- \eta,0\right\}\right] \leq\frac{2b}{N^{1.5}\log N}+\frac{5\mu_{\max}+\lambda}{\sqrt{N}\log N} \leq \frac{5\mu_{\max}+2}{\sqrt{N}\log N}. \nonumber
\end{align} 

\subsection{Proving Theorem \ThmmainB} \label{sec:0-delay}
Once we have Lemma \ref{lem:main}, we can prove Theorem \ref{Thm:main} with the property of LB-zero and the Markov inequality. 
For an LB-zero policy in $\Pi,$ the waiting probability satisfies 
\begin{align*}
\mathbb P(\mathcal W)=& \mathbb P\left(\mathcal W|\sum_{i=1}^bS_i\leq 1 - \frac{1}{N^\alpha\log N}\right)\mathbb P\left(\sum_{i=1}^bS_i\leq 1 -\frac{1}{N^\alpha\log N}\right) \\
&+ \mathbb P\left(\mathcal W|\sum_{i=1}^bS_i> 1 -\frac{1}{N^\alpha\log N}\right)\mathbb P\left(\sum_{i=1}^bS_i > 1 - \frac{1}{N^\alpha\log N}\right) \\
\leq& \mathbb P\left(\mathcal W|\sum_{i=1}^bS_i< 1 - \frac{1}{N^\alpha\log N}\right) + \mathbb P\left(\sum_{i=1}^bS_i > 1 - \frac{1}{N^\alpha\log N}\right) 
\end{align*}
where the first term is bounded by $\frac{1}{\sqrt{N}}$ because of the definition of LB-zero. For the second term, we have   
\begin{align*}
\mathbb P\left(\sum_{i=1}^bS_i > 1 - \frac{1}{N^\alpha\log N}\right) \leq& \mathbb P\left(\max\left\{\sum_{i=1}^bS_i-\lambda -\frac{k\log N}{\sqrt{N}}, 0\right\}> \frac{1}{N^{\alpha}}\left(1-\frac{1}{\log N}\right)-\frac{k\log N}{\sqrt{N}}\right)\\
\leq& \frac{\mathbb E \left[ \max\left\{\sum_{i=1}^bS_i-\lambda -\frac{k\log N}{\sqrt{N}}, 0\right\} \right] }{\frac{1}{N^{\alpha}}\left(1-\frac{1}{\log N}\right)-\frac{k\log N}{\sqrt{N}}}\\
\leq& \frac{\mathbb E \left[ \max\left\{\sum_{i=1}^bS_i-\lambda -\frac{k\log N}{\sqrt{N}}, 0\right\}\right]}{\frac{1}{2N^{\alpha}}}\\
\leq&\frac{10\mu_{\max}+4}{N^{0.5-\alpha} \log N}
\end{align*}
where the second inequality holds because of the Markov inequality; the third inequality holds because $\log N \geq 2$ and $\frac{N^{0.5-\alpha}}{\log N}\geq 2k;$ and the last inequality holds because of Lemma \ref{lem:main}. 

Finally, we remark that we choose $k=\Omega(b)$ to prove Lemma \ref{SSC:s1 s2},  which is the technical reason we assumed $b$ is finite. We, however, believe our results hold even for $b=\infty.$  

\section{Simulations}
In this section, we confirm our theoretical results of ISSP in Theorem \ref{Thm:equilibrium} and the large system insensitivity in Theorem \ref{Thm:main} with simulations. We considered two policies, JSQ and JIQ, and $\lambda = 1 - N^{\alpha}$ 
We conjecture that our results hold even for $\alpha=0.5$ because it holds for any $\alpha<0.5.$ To confirm this, we used $\alpha =0.5$ in our simulations. 

\subsection{Large System Insensitivity under JSQ and JIQ}
We first studied the average total queue length (per server) $\mathbb E[\sum_{i=1}^b S_{i}]$ and the waiting probability $\mathbb P(\mathcal W)$ under a Coxian-$4$ service time distribution with the parameters $p=[0.5, 0.5, 0.5, 1.0]$ and $\mu=[1.875, 1.875, 1.875, 1.875].$ We plotted $\mathbb E[\sum_{i=1}^b S_{i}]$ and  $\mathbb P(\mathcal W)$ versus the number of servers $N.$ The results are obtained with $10$ trials and each trial has $10^7$ steps. From Figure \ref{fig:verse N}, the waiting probability tends to zero when $N$ increases as we expected and JIQ almost has the identical performance as JSQ. 
\begin{figure}[H]
     \centering
     \begin{subfigure}[b]{0.49\textwidth}
         \centering
         \includegraphics[width=\textwidth]{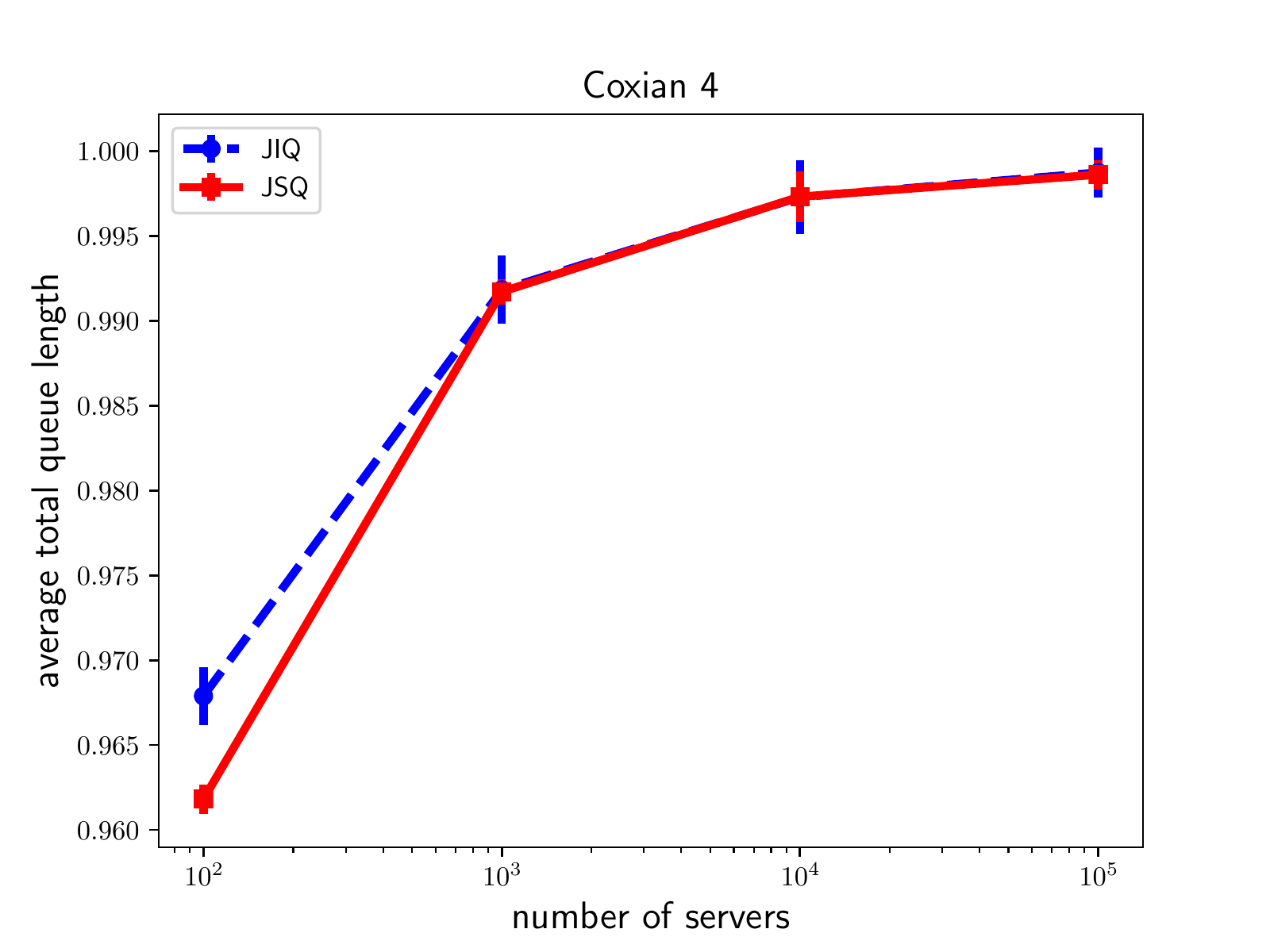}
         \caption{Average total queue length under Coxian $4$}
     \end{subfigure}
     \hfill
     \begin{subfigure}[b]{0.49\textwidth}
         \centering
         \includegraphics[width=\textwidth]{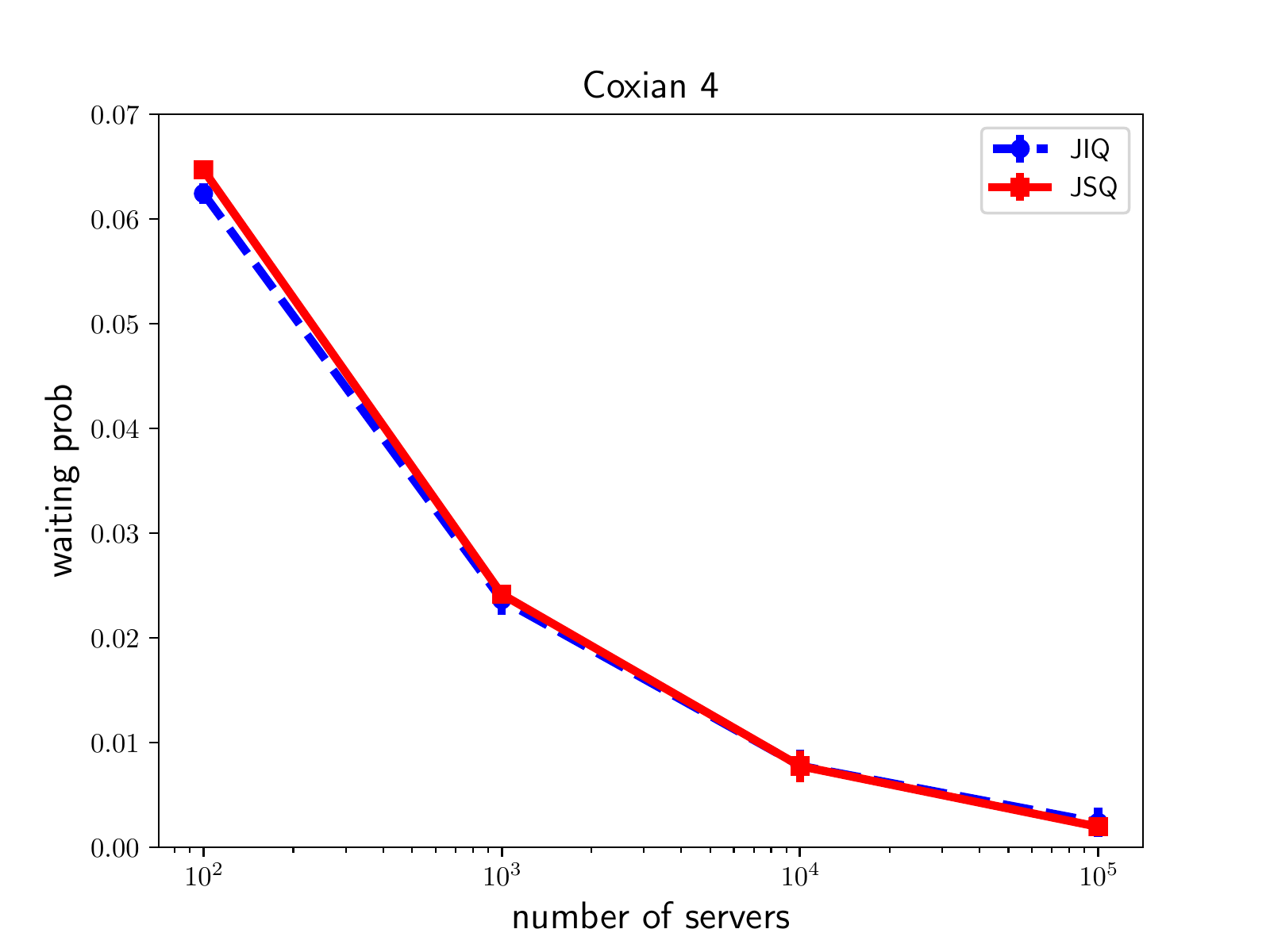}
         \caption{Waiting probability under Coxian $4$}
     \end{subfigure}
     \caption{Asymptotic zero waiting under JSQ and JIQ}
     \label{fig:verse N}
\end{figure}
\begin{figure}[H]
     \centering
     \begin{subfigure}[b]{0.49\textwidth}
         \centering
         \includegraphics[width=\textwidth]{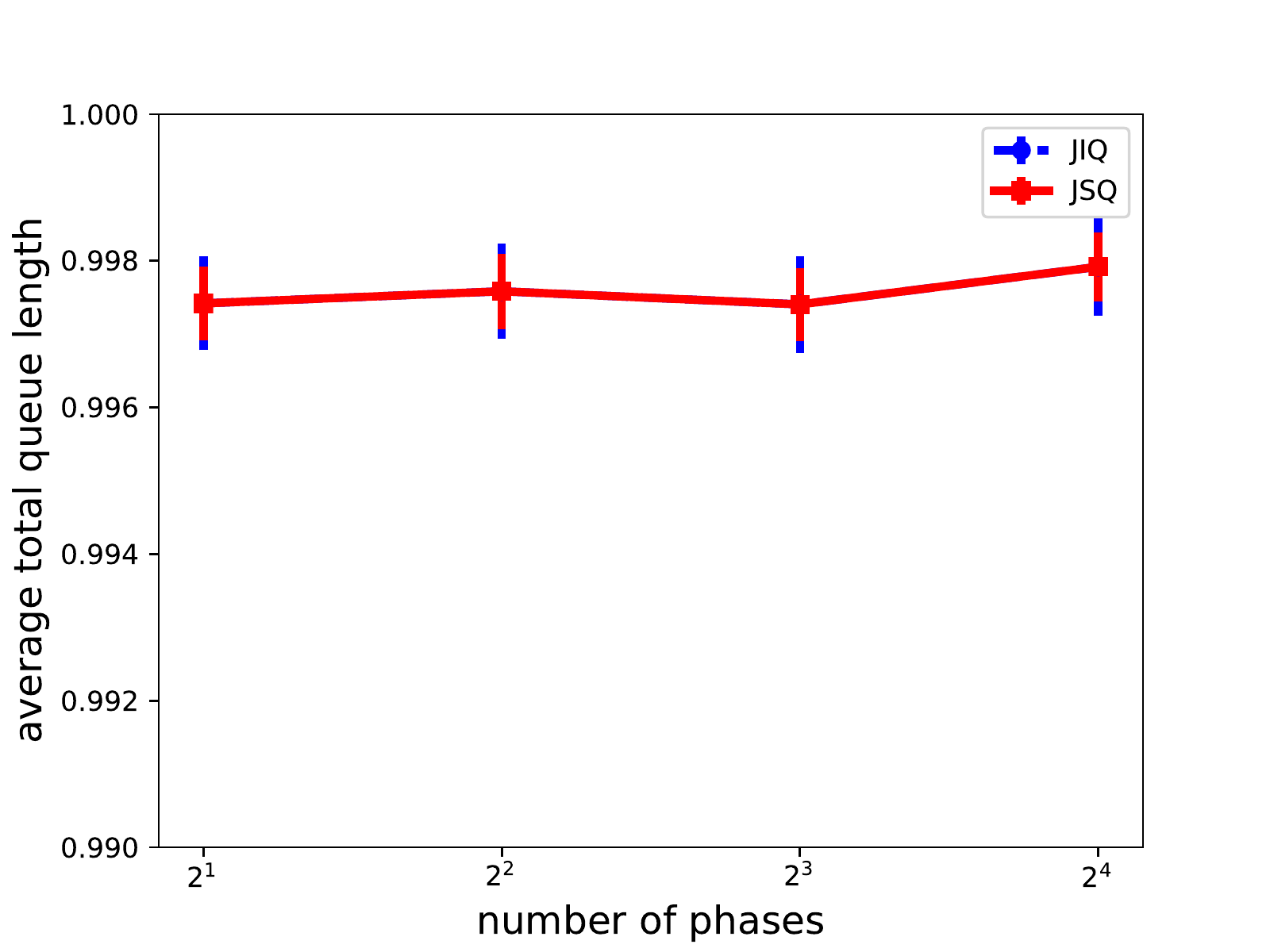}
         \caption{Average total queue length under Coxian $M$}
     \end{subfigure}
     \hfill
     \begin{subfigure}[b]{0.49\textwidth}
         \centering
         \includegraphics[width=\textwidth]{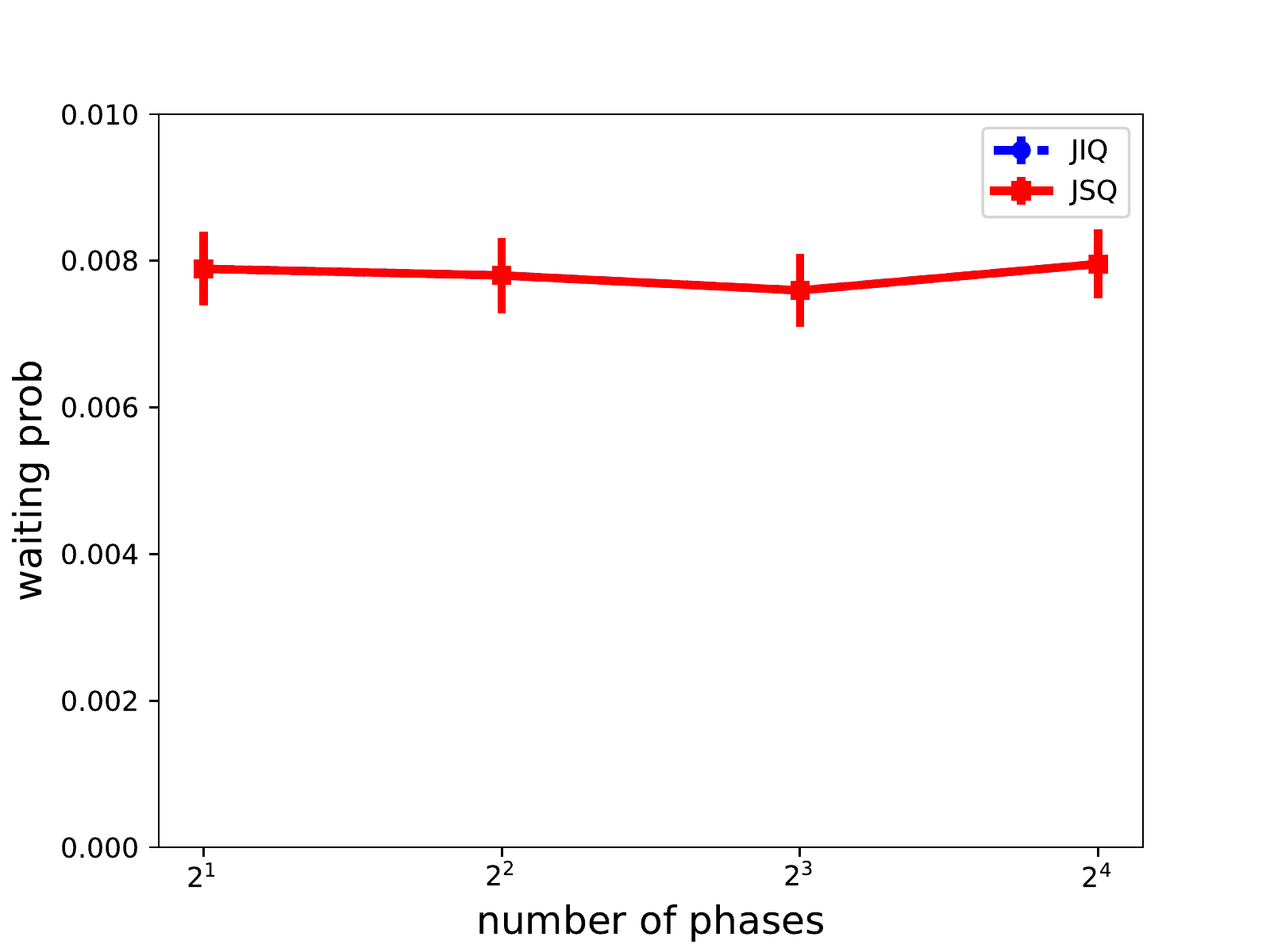}
         \caption{Waiting probability under Coxian $M$}
     \end{subfigure}
     \caption{Large System Insensitivity under JSQ and JIQ}
     \label{fig:verse phase M}
\end{figure}
We then investigated $\mathbb E[\sum_{i=1}^b S_{i}]$ and $\mathbb P(\mathcal W)$ versus Coxian-$M$ with various number of phases $M$ and fixed $N=10^4.$ In particular, we consider Coxian-$M$ with $p=[p_1, p_2, \cdots, 1]$ and $\mu=[\mu_1, \mu_2, \cdots, \mu_M],$ where $p_m = 0.5, \forall 1 \leq m \leq M-1$ and $\bar \mu= \mu_m, \forall m$ (identical service times). We plotted $\mathbb E[\sum_{i=1}^b S_{i}]$ and  $\mathbb P(\mathcal W)$ versus $M.$ From Figure \ref{fig:verse phase M}, we can observe that the average queue length and the waiting probability remain roughly the same under different $M$s, which confirms the insensitivity.

\subsection{ISSP under JSQ and JIQ}
In this section, we investigated ISSP by studying the trajectory of $S_{1,m}(t), \forall m$ under JSQ and JIQ. We considered $N=10,000$ and a Coxian-$4$ service time distribution with $p=[0.5, 0.5, 0.5, 1.0]$ and $\mu=[1.875, 1.875, 1.875, 1.875].$ 
We plotted $S_{1,m}(t)$ of JSQ and JIQ in Figure \ref{fig:ISSP}. The results are obtained with $10$ trials and each trial has $10^6$ steps. We observed that $S_{1,m(t)}$ under both policies concentrates around dash lines $s_{1,m}^*, \forall m,$ which confirms the high probability bounds in \ThmmainA. Note that the system was initialized with the zero-waiting equilibrium, instead of the empty state, in our simulations. 
\begin{figure}[H]
     \centering
     \begin{subfigure}[b]{0.49\textwidth}
         \centering
         \includegraphics[width=\textwidth]{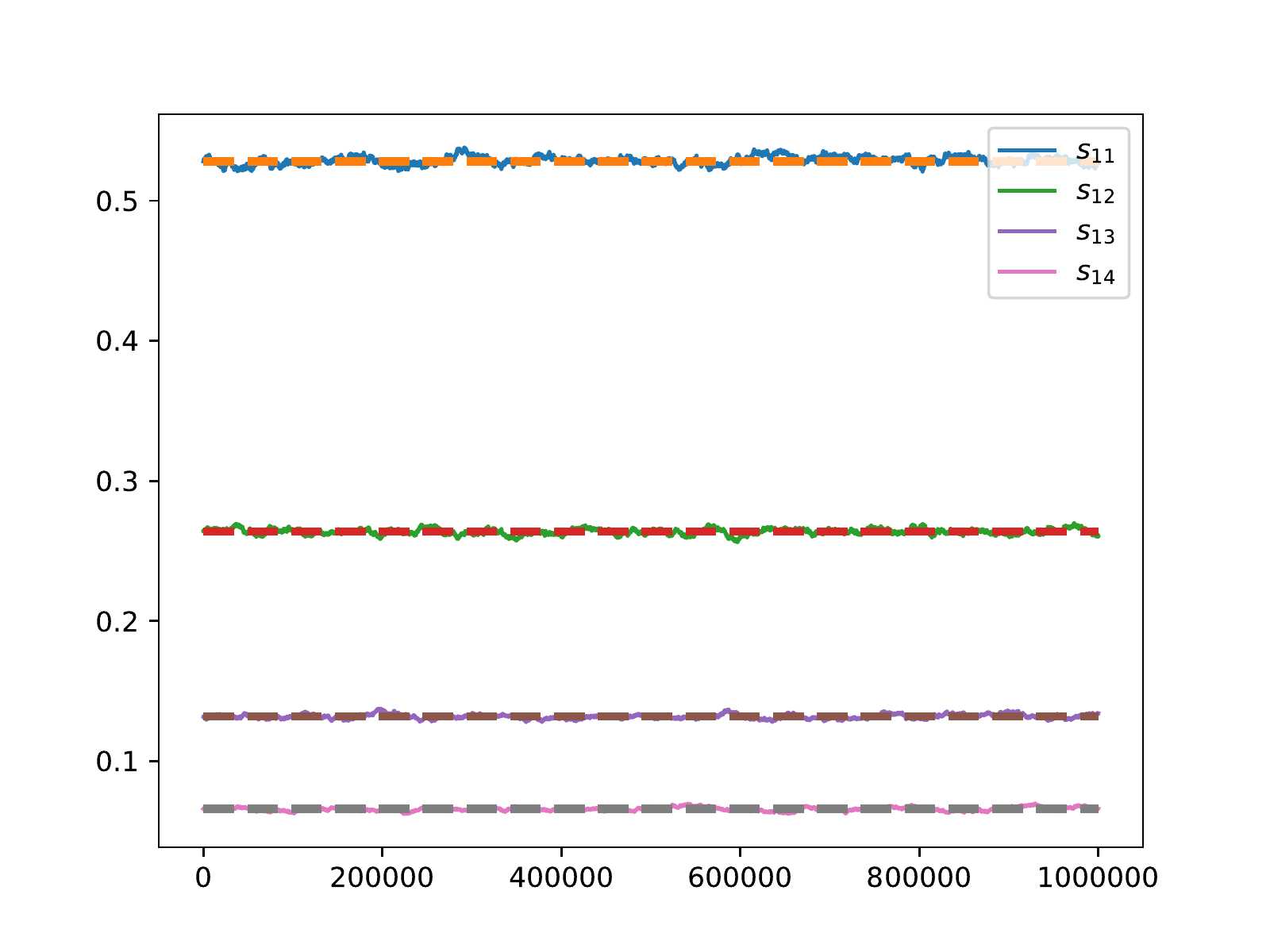}
     \caption{Trajectories under JSQ}
     \end{subfigure}
     \hfill
     \begin{subfigure}[b]{0.49\textwidth}
         \centering
         \includegraphics[width=\textwidth]{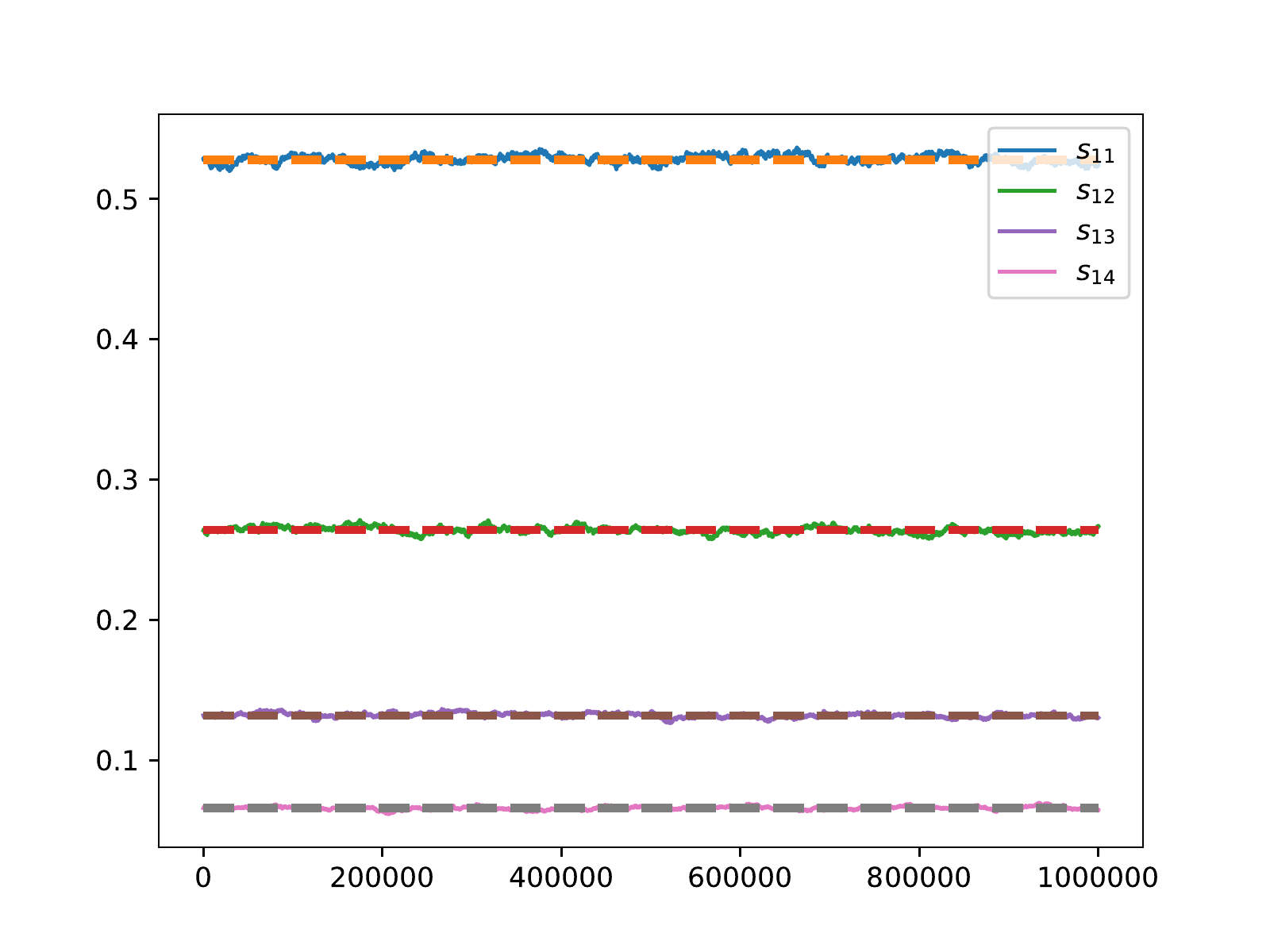}
     \caption{Trajectories under JIQ}
     \end{subfigure}
     \caption{The evolution of the system states under JSQ and JIQ}
     \label{fig:ISSP}
\end{figure}

\section{Conclusions and Discussions}
In this paper, we studied a distributed queueing system under Coxian service time distributions in the sub-Halfin-Whitt regime. We established that a set of load balancing policies, named LB-zero, achieves asymptotic zero-waiting, i.e. insensitive in the large-system regime. To tackle the non-monotonicity under general service time distributions, we developed a technique, called  iterative state-space peeling (ISSP), which iteratively removes the low-probability states, and results in a small state-space that can be analyzed using a simple mean-field model. This ISSP approach may be used for other problems as well. One possible application is to study load-balancing in many server systems with heterogeneous servers or jobs belonging to multiple priority classes. For heterogeneous servers, we can use ISSP to identify the ``typical'' load of each type of the servers; and for jobs with different priorities, we can use ISSP to identify the ``typical'' distribution of job types in the system. In the reduced state space based on the typical load of the typical distribution at the steady-state, the steady-state performance of the many server system may become tractable like in this paper.

\bibliographystyle{abbrv}
\bibliography{inlab-refs}%

\appendix

\section{The Properties of the Constants}\label{app:constants}
The following two lemmas show that the constants $a_m,$ $b_m,$ $c_m,$ ($\forall m \geq 2,$) and $C_M$ are positive, and $0<\xi<1.$ 
\begin{lemma}\label{lem:positive constants}
The constants $a_m,b_m,c_m, \forall m \geq 2,$ and $C_M$ are positive.
\end{lemma}

\begin{proof}
$1 < a_m < 1, \forall 2\leq m \leq M$ holds by the definition. Therefore, it is easy to verify $c_m, \forall m \geq 2,$ and $C_M$ are positive.

Next, we prove $b_m, \forall m \geq 2,$ is positive as follows:
\begin{align}
b_m =& (1-a_m)\left(1 + \sum_{r=m+1}^M \frac{v_{r}}{v_{1}}\right) - \frac{a_mv_m}{v_1} \nonumber\\
=& 1- a_m - \frac{a_mv_m}{v_1} + (1-a_m)\sum_{r=m+1}^M \frac{v_{r}}{v_{1}} \nonumber\\
=& \frac{p_1 \mu_1}{p_1\mu_1 + \mu_m} \left(1 - \prod_{i=2}^{m-1}p_i\right) + (1-a_m)\sum_{r=m+1}^M \frac{v_{r}}{v_{1}} \nonumber\\
>& 0 \nonumber
\end{align}
where the first equality holds by the definition of $b_m;$ the third equality by substituting the definition of $a_m$ and $v_m;$ the last inequality holds because $\prod_{i=2}^{m-1}p_i \leq 1, \forall m \geq 2$ and $0 < a_m < 1, \forall m \geq 2.$
\end{proof}

\begin{lemma}\label{lem:xi}
$1-\xi = \mu_1\prod_{m=2}^M a_m.$
\end{lemma}
\begin{proof}
Recall the definition of $\xi=\sum_{m=2}^{M} b_m  \prod_{j=m+1}^M a_j.$ We have
\begin{align}
1- \sum_{m=2}^{M} b_m \prod_{j=m+1}^M a_j 
=& 1- \sum_{m=2}^{M} \left((1-a_m)\left(1 + \sum_{r=m+1}^M \frac{v_{r}}{v_{1}}\right) - a_m\frac{v_m}{v_1}\right) \prod_{j=m+1}^M a_j \nonumber\\
=& 1- \sum_{m=2}^{M} (1-a_m) \prod_{j=m+1}^M a_j - \sum_{m=2}^{M} \left((1-a_m)\sum_{r=m+1}^M \frac{v_{r}}{v_{1}} - a_m\frac{v_m}{v_1}\right) \prod_{j=m+1}^M a_j \nonumber\\
=& \prod_{j=2}^M a_j - \sum_{m=2}^{M} \left(\sum_{r=m+1}^M \frac{v_{r}}{v_{1}} - a_m\sum_{r=m}^M\frac{v_r}{v_1}\right) \prod_{j=m+1}^M a_j \nonumber\\
=& \prod_{j=2}^M a_j - \sum_{m=2}^{M} \sum_{r=m+1}^M \frac{v_{r}}{v_{1}}  \prod_{j=m+1}^M a_j + \sum_{m=2}^{M} \sum_{r=m}^M\frac{v_r}{v_1} \prod_{j=m}^M a_j \nonumber\\
=& \prod_{j=2}^M a_j + \sum_{m=2}^{M} \frac{v_m}{v_1} \prod_{j=2}^M a_j = \prod_{j=2}^M a_j + \frac{1-v_1}{v_1} \prod_{j=2}^M a_j = \mu_1\prod_{m=2}^M a_m. \nonumber
\end{align}
\end{proof}

\section{Proof of the Lemmas for Theorem \ThmmainA} \label{App:IterSSC}

We first prove Lemmas  \ref{lem: iter head tail}, \ref{lem: iter prob lower bound}, \ref{lem: iter prob upper bound}, and \ref{lem: L11 lift} used in ISSP. 

\subsection[]{A tail bound from \cite{WanMagSri_21}}

We introduce Lemma \ref{tail-bound-cond} from \cite{WanMagSri_21}, which is an extension of the tail bound in \cite{BerGamTsi_01} and is the key to establish ISSP. Lemma \ref{tail-bound-cond} allows us to apply the Lyapunov drift analysis to iteratively reduce the state space. 

\begin{lemma}\label{tail-bound-cond}
Let $(S(t): t \geq 0)$ be a continuous-time Markov chain over a finite state space $\mathcal S$ and is irreducible, so it has a unique stationary distribution $\pi.$  Consider a Lyapunov function $V: \mathcal S \to R^{+}$ and define the drift of $V$ at a state $s \in \mathcal S$ as $$\nabla V(s) = \sum_{s' \in \mathcal S: s' \neq s} q_{s,s'} (V(s') - V(s)),$$ where $q_{s,s'}$ is the transition rate from $s$ to $s'.$ Assume
\begin{align*} 
\nu_{\max} :=& \max\limits_{s,s'\in \mathcal S: q_{s,s'} >0} |V(s') - V(s)|< \infty ~~ \text{ and } ~~ \bar q := \max\limits_{s \in \mathcal S} (-q_{s,s}) < \infty 
\end{align*}
and define $$q_{\max}:=\max\limits_{s \in \mathcal S} \sum_{s' \in \mathcal S: V(s) < V(s')} q_{s,s'}.$$
Assume there exists a set $\mathcal E$ with $B>0$, $\gamma>0$, $\delta \geq 0$ such that the following conditions hold
\begin{enumerate}[(i)]
\item $\nabla V(s) \leq -\gamma$ when $V(s) \geq B$ and $s \in \mathcal E.$
\item $\nabla V(s) \leq \delta$ when $V(s) \geq B$ and $s \notin \mathcal E.$
\end{enumerate}
Then $$\mathbb P \left(V(S)\geq B + 2 \nu_{\max} j \right) \leq \alpha^j + \beta \mathbb P\left(S \notin \mathcal E\right), ~ \forall j \in \mathbb{N},$$
with $$\alpha = \frac{q_{\max}\nu_{\max}}{q_{\max}\nu_{\max} + \gamma} ~~\text{ and }~~ \beta = \frac{\delta}{\gamma}+1.$$
\end{lemma}

According to Lemma \ref{tail-bound-cond}, the critical step in establishing the tail bound is to construct proper Lyapunov functions. In the following sections, we construct a sequence of Lyapunov functions and apply Lemma \ref{tail-bound-cond} to prove Lemmas \ref{lem: iter head tail}, \ref{lem: iter prob lower bound}, \ref{lem: iter prob upper bound},  and \ref{lem: L11 lift}. In the following proofs, we ignore the iteration number $n$ for a clean notation. 
 
\subsection{Proof of Lemma \ref{lem: iter head tail}: A lower bound on $L_{1,1}(n+1)$ given $\sum_{m=2}^{M} S_{1,m} \leq U_{M}(n)$}
\iterheadtail*

To prove Lemma \ref{lem: iter head tail} using Lemma \ref{tail-bound-cond}, we consider the following Lyapunov function
\begin{align}
V(s) = \tilde{L}_{1,1} - s_{1,1}, \label{Ly:s11}
\end{align}  where $\tilde{L}_{1,1} = \min\left\{1 - \frac{1-\xi}{2\mu_1N^\alpha} - U_M, s_{1,1}^*\right\}$ and define $$\mathcal E = \left\{s ~\left|~ \sum_{m=2}^{M} s_{1,m} \leq U_{M}\right.\right\}\quad \hbox{ and } \quad B=\frac{2\Delta}{C}.$$

When $V(s) = \tilde{L}_{1,1} - s_{1,1} \geq B$ and $s\in \mathcal E,$ 
we have $$s_1 = \sum_{m=1}^M s_{1,m} \leq U_M + \tilde{L}_{1,1} - \frac{2\Delta}{C} = 1-\frac{1-\xi}{2\mu_1N^\alpha}- \frac{2\Delta}{C} \leq 1 - \frac{1}{N^\alpha\log N},$$
where the last inequality holds due to $\log N \geq \frac{2\mu_1}{1-\xi}.$
Therefore, the drift of $V(s)$ satisfies
\begin{align*} 
\nabla V(s) =& -\lambda (1-A_1(s)) + \mu_1s_{1,1} - \sum_m (1-p_m)\mu_ms_{2,m}\\
\stackrel{(a)}{\leq}&  \frac{1}{\sqrt{N}}-\lambda  + \mu_1s_{1,1}\\
\stackrel{(b)}{\leq}&  \frac{1}{\sqrt{N}}-\lambda  + \mu_1\left(\tilde{L}_{1,1} - \frac{2\Delta}{C}\right)\\
=&  \frac{1}{\sqrt{N}}- (\lambda - \mu_1 \tilde{L}_{1,1}) -\frac{2\mu_1\Delta}{C}\\
\stackrel{(c)}{\leq}&  \frac{1}{\sqrt{N}} -\frac{2\mu_1\Delta}{C}\\
\stackrel{(d)}{\leq}& -\frac{\mu_1\Delta}{C}
\end{align*} where
\begin{itemize}
\item $(a)$ holds because $A_{1}(s) \leq \frac{1}{\sqrt{N}}$ when $s_1\leq  1 - \frac{1}{N^\alpha\log N}$ for a LB-zero policy in $\Pi$ and $s_{2,m} \geq 0;$
\item $(b)$ holds because $V(s) = \tilde{L}_{1,1} - s_{1,1} \geq \frac{2\Delta}{C};$
\item $(c)$ holds because $\tilde{L}_{1,1} \leq s_{1,1}^* = \frac{\lambda}{\mu_1};$
\item and $(d)$ holds because $\log N \geq \frac{C}{\mu_1}.$
\end{itemize} 

Moreover, we have 
$$\nabla V(s) = -\lambda (1-A_1(s)) + \mu_1s_{1,1} - \sum_m (1-p_m)\mu_ms_{2,m} \leq \mu_1.$$ 

Define $\gamma =  \frac{\mu_1\Delta}{C}$ and $\delta = \mu_1.$ We now apply Lemma \ref{tail-bound-cond} with $j=\frac{2\sqrt{N}\log N}{C}.$ Since $q_{\max} = \mu_{1}N$ and $\nu_{\max} = \frac{1}{N},$ we have $$\alpha = \frac{1}{1 + \frac{\Delta}{C}} ~~\text{ and }~~ \beta =\frac{C}{\Delta}+1,$$ 
and
\begin{align*}
\mathbb P\left(S_{1,1} <L_{1,1}(n+1)\right)\leq& \mathbb P\left(S_{1,1} \leq L_{1,1}(n+1)\right)\\
\stackrel{(a)}{=} & \mathbb P\left(V(S) \geq B + 2 \nu_{\max} j\right) \\
\stackrel{(b)}{\leq}& \left(\frac{1}{1 + \frac{\Delta}{C}}\right)^{\frac{2\sqrt{N}\log N}{C}} + \beta \mathbb P\left(S \notin {\mathcal E}\right)\\
\stackrel{(c)}{\leq}& \left(1 - \frac{\Delta}{2C} \right)^{\frac{2\sqrt{N}\log N}{C}} + \beta \sigma_{M}\\
\leq& e^{-\frac{\log^2 N}{C^2}} +\beta \sigma_{M}
\end{align*} where
\begin{itemize}
\item $(a)$ holds by substituting $B = \frac{2\Delta}{C},$ $\nu_{\max}=\frac{1}{N}$ and $j=\frac{2\sqrt{N}\log N}{C};$
\item $(b)$ holds based on Lemma \ref{tail-bound-cond};
\item and $(c)$ holds because $\frac{1}{C}\leq \frac{1}{\Delta}$ and the assumption of the lemma on $ \mathbb P\left(S \notin {\mathcal E}\right).$
\end{itemize}

\subsection{Proof of Lemma  \ref{lem: iter prob lower bound}: A lower bound on $S_{1,m}$ given $S_{1,m-1} \geq L_{1,m-1}$ for $m \geq 2$}
\iterproblowerbound*

To prove Lemma \ref{lem: iter prob lower bound}, consider Lyapunov function 
\begin{align}
V(s) = \frac{v_m}{v_{m-1}} L_{1,m-1} - s_{1,m}. \label{Ly:s1m lower}
\end{align} and define $$\mathcal E = \left\{s ~|~ s_{1,m-1} \geq L_{1,m-1}\right\}.$$

Given $V(s) \geq \frac{v_m}{C} \Delta,$ 
we have $$s_{1,m} \leq \frac{v_m}{v_{m-1}} L_{1,m-1} - \frac{v_m}{C} \Delta.$$
Therefore, the drift of $V(s)$ when $V(s) \geq \frac{v_m}{C} \Delta$ and $s\in \mathcal E$ is
\begin{align*} 
\nabla V(s) =& \mu_m s_{1,m} - p_{m-1}\mu_{m-1} s_{1,m-1}\\
\stackrel{(a)}{=}& \mu_m \left(s_{1,m} - \frac{v_m}{v_{m-1}} s_{1,m-1}\right)\\
\stackrel{(b)}{\leq}&  \mu_m \left(s_{1,m} - \frac{v_m}{v_{m-1}} L_{m-1}\right)\\
\stackrel{(c)}{\leq}& -\frac{\mu_m v_m}{C} \Delta
\end{align*} where
\begin{itemize}
\item $(a)$ holds according to the definition of $v_m = \frac{\prod_{i=1}^{m-1}p_i}{\mu_m};$
\item $(b)$ holds because $s_{1,m-1} \geq L_{1,m-1};$
\item and $(c)$ holds because $s_{1,m} \leq \frac{v_m}{v_{m-1}}L_{1,m-1} - \frac{v_m}{C} \Delta.$
\end{itemize}

Moreover, we have 
$$\nabla V(s) = \mu_m s_{1,m} - p_{m-1}\mu_{m-1} s_{1,m-1} \leq \mu_m.$$ 

Define $B = \frac{v_m}{C}\Delta$, $\gamma =  \frac{\mu_m v_m}{C}\Delta,$ and $\delta = \mu_m.$ Combining $q_{\max} = \mu_{m}N$ and $\nu_{\max} = \frac{1}{N},$ we have $$\alpha = \frac{1}{1 + \frac{v_m}{C}\Delta} ~~\text{ and }~~ \beta = \frac{C}{v_m \Delta}+1.$$ 

Applying Lemma \ref{tail-bound-cond} with $j=\frac{2v_m\sqrt{N}\log N}{C},$ we have
\begin{align*}
\mathbb P\left(S_{1,m} < L_{1,m}(n+1)\right) \stackrel{(a)}{\leq}&  \mathbb P\left(V(S) \geq B + 2 \nu_{\max} j\right)\\
\stackrel{(b)}{\leq}& \left(\frac{\mu_m}{\mu_m + \frac{\mu_m v_m}{C}\Delta}\right)^{\frac{2v_m\sqrt{N}\log N}{C}} + \beta \mathbb P\left(S \notin {\mathcal E}\right)\\
\stackrel{(c)}{\leq}& \left(1 - \frac{v_m}{2C}\Delta\right)^{\frac{2v_m\sqrt{N}\log N}{C}} + \beta \epsilon_{m-1}\\
\leq& e^{-\frac{v_m^2\log^2 N}{C^2}} +\beta \epsilon_{m-1},
\end{align*} where
\begin{itemize}
\item $(a)$ holds by substituting $B = \frac{v_m}{C}\Delta,$ $\nu_{\max}=\frac{1}{N}$ and $j=\frac{2v_m\sqrt{N}\log N}{C};$
\item $(b)$ holds based on Lemma \ref{tail-bound-cond};
\item and $(c)$ holds because $\frac{v_m}{C} \leq \frac{1}{\Delta}.$
\end{itemize}


\subsection{Proof of Lemma \ref{lem: iter prob upper bound}: An upper bound on $\sum_{r=2}^{m} S_{1,r}$ given $\sum_{r=2}^{m-1} S_{1,r} \leq U_{m-1}, \forall m \geq 2$ and $S_{1,m} \geq L_{m}, \forall m \geq 1$}
\iterprobupperbound*

Consider Lyapunov function 
\begin{align}
V(s) = \sum_{r=2}^{m} s_{1,r} - B_m, \label{Ly:s1m upper}
\end{align} 
where $$B_m = \frac{p_1\mu_1 (1-\sum_{r=m+1}^ML_{1,r}) - \sum_{r=2}^{m-1} (1-p_r)\mu_r L_{1,r} + \mu_m U_{m-1}}{p_1\mu_1 + \mu_m},$$ and define $$\mathcal E = \left\{s ~\left|~ \sum_{r=2}^{m-1} s_{1,r} \leq U_{m-1},\hbox{ and }~ s_{1,r} \geq L_{1,r}, \forall r \geq 1\right.\right\}.$$ 

Given $V(s) \geq \frac{p_1\mu_1}{p_1\mu_1 + \mu_m}\frac{\Delta}{C}$ and $s\in \mathcal E,$ we have 
\begin{align*}
\nabla V(s) 
=& \sum_{r=2}^m \left( p_{r-1}\mu_{r-1} s_{1,r-1} - \mu_r s_{1,r} \right)\\
=& p_1 \mu_1 s_{1,1} - \mu_m s_{1,m} - \sum_{r=2}^{m-1} (1-p_r) \mu_r s_{1,r}\\
\stackrel{(a)}{\leq}& p_1 \mu_1 - p_1 \mu_1 \sum_{r=2}^M s_{1,r} - \mu_m s_{1,m} - \sum_{r=2}^{m-1} (1-p_r) \mu_r s_{1,r}\\
=& p_1 \mu_1 - p_1 \mu_1 \sum_{r=2}^M s_{1,r} - \mu_m \sum_{r=2}^m s_{1,r} + \mu_m \sum_{r=2}^{m-1} s_{1,r} - \sum_{r=2}^{m-1} (1-p_r) \mu_r s_{1,r}\\
=& p_1 \mu_1 - \left(p_1 \mu_1 + \mu_m\right) \sum_{r=2}^m s_{1,r} - p_1 \mu_1\sum_{r=m+1}^M s_{1,r} - \sum_{r=2}^{m-1} (1-p_r) \mu_r s_{1,r} + \mu_m \sum_{r=2}^{m-1} s_{1,r}\\
\stackrel{(b)}{\leq}& p_1 \mu_1 \left(1 - \sum_{r=m+1}^M L_{1,r}\right) - \left(p_1 \mu_1 + \mu_m\right)B_{m}  - \sum_{r=2}^{m-1} (1-p_r) \mu_r L_{1,r} + \mu_m U_{m-1} - \frac{p_1\mu_1\Delta}{C}\\
\stackrel{(c)}{=}& -\frac{p_1\mu_1\Delta}{C}, 
\end{align*}
where 
\begin{itemize}
    \item $(a)$ holds because $s_{1,1} = s_1 - \sum_{r=2}^M s_{1,r}$ and $s_1 \leq 1;$
    \item $(b)$ holds because $s_{1,r} \geq L_{1,r}$ for any $1\leq r\leq M,$ $\sum_{r=2}^{m-1} s_{1,r} \leq U_{m-1}$ and $\sum_{r=2}^{m} s_{1,r} \geq B_m + \frac{p_1\mu_1}{p_1\mu_1 + \mu_m}\frac{\Delta}{C}$ implied by $V(s) \geq \frac{p_1\mu_1}{p_1\mu_1 + \mu_m}\frac{\Delta}{C};$
    \item and $(c)$ holds by the definition of $B_m = \frac{p_1\mu_1 (1-\sum_{r=m+1}^ML_{1,r}) - \sum_{r=2}^{m-1} (1-p_r)\mu_r L_{1,r} + \mu_m U_{m-1}}{p_1\mu_1 + \mu_m}.$ 
\end{itemize}

Moreover, we have
\begin{align*}
\nabla V(s) 
= p_1 \mu_1 s_{1,1} - \mu_m s_{1,m} - \sum_{r=2}^{m-1} (1-p_r) \mu_r s_{1,r} \leq p_1\mu_1 s_{1,1} \leq p_1\mu_{1}. \end{align*}

We now apply Lemma \ref{tail-bound-cond} with $j=\frac{2\sqrt{N}\log N}{C}.$  Define $B = \frac{p_1\mu_1}{p_1\mu_1 + \mu_r}\frac{\Delta}{C}$, $\gamma =  \frac{p_1\mu_1\Delta}{C},$ and $\delta = p_1\mu_1.$ Since $q_{\max} = p_1\mu_{1}N$ and $\nu_{\max} = \frac{1}{N},$ we have $$\alpha = \frac{1}{1 + \frac{\Delta}{C}} ~~\text{ and }~~ \beta = \frac{C}{\Delta}+1,$$ 
and
\begin{align*}
\mathbb P\left(V(S) \geq B + 2 \nu_{\max} j\right) 
\stackrel{(a)}{=}& \mathbb P\left(\sum_{m=2}^rS_{1,m}-B_r \geq \frac{p_1\mu_1}{p_1\mu_1 + \mu_m}\frac{\Delta}{C} + \frac{4\Delta}{C}\right)\\
\stackrel{(b)}{\leq}& \left(\frac{1}{1 + \frac{\Delta}{C}}\right)^{\frac{2\sqrt{N}\log N}{C}} + \beta \mathbb P\left(S \notin {\mathcal E}\right)\\
\stackrel{(c)}{\leq}& \left(1 - \frac{\Delta}{2C} \right)^{\frac{2\sqrt{N}\log N}{C}} + \beta \left(\sigma_{m-1} + \sum_{m=1}^M \epsilon_m\right)\\
\leq& e^{-\frac{\log^2 N}{C^2}} +\beta \left(\sigma_{m-1} + \sum_{m=1}^M \epsilon_m\right)
\end{align*} where
\begin{itemize}
\item $(a)$ holds by substituting $B = \frac{p_1\mu_1}{p_1\mu_1 + \mu_m}\frac{\Delta}{C},$ $\nu_{\max}=\frac{1}{N}$ and $j=\frac{2\sqrt{N}\log N}{C};$
\item $(b)$ holds based on Lemma \ref{tail-bound-cond};
\item amd $(c)$ holds because $\frac{1}{C} \leq \frac{1}{\Delta}$ and union bounds on $\mathbb P(S \notin \mathcal E).$
\end{itemize}

Now we prove $U_{m} = B_m +\left(\frac{p_1\mu_1}{p_1\mu_1 + \mu_m} + 4\right)\frac{\Delta}{C},$ which serves the upper bound on $\sum_{r=2}^m S_{1,r}$ and we represent $U_m$ with $L_{1,1}.$  
Recall the definition of $L_{1,m}$ from the previous subsection that
\begin{align}\label{eq:Lm on Lm-1}  
L_{1,m} = \frac{v_m}{v_{m-1}} L_{1,m-1} - \frac{5v_m}{C}\Delta, \forall m \geq 2,  
\end{align}
which implies that
\begin{align}\label{eq:Lm on L1}
L_{1,m} = \frac{v_m}{v_{1}} L_{1,1} - \frac{5(m-1)v_m}{C}\Delta, \forall m \geq 2.
\end{align}
Therefore, we have 
\begin{align}
\sum_{r=m+1}^M L_{1,r} =& \sum_{r=m+1}^M \frac{v_r}{v_1}L_{1,1} - \sum_{r=m+1}^M \frac{5(r-1)v_r}{C}\Delta,
\end{align}
and
\begin{align}
\sum_{r=2}^{m-1} (1-p_r)\mu_r L_{1,r} 
=& \sum_{r=2}^{m-1} p_{r-1}\mu_{r-1} L_{1,r-1} - p_r\mu_r L_{1,r} - \frac{5\mu_rv_r}{C}\Delta \nonumber\\
=& p_1\mu_1 - \mu_{m} L_{1,m} + \frac{5\mu_mv_m}{C}\Delta - \sum_{r=2}^{m-1}\frac{5\mu_rv_r}{C}\Delta \nonumber\\
=& \left(p_1\mu_1 - \frac{\mu_m v_m}{v_1}\right)L_{1,1} - \sum_{r=2}^{m-1} \frac{5\mu_r v_r}{C}\Delta + \frac{5(m-2)\mu_m v_m}{C}\Delta
\end{align}
where the first and second equalities hold by substituting \eqref{eq:Lm on Lm-1} and the last equality holds by substituting \eqref{eq:Lm on L1}. Finally we have
\begin{align}
& B_m + \left(\frac{p_1\mu_1}{p_1\mu_1+\mu_m}+4\right)\frac{\Delta}{C} \\
=& \frac{p_1\mu_1 (1-\sum_{r=m+1}^ML_r) - \sum_{r=2}^{m-1} (1-p_r)\mu_r L_{1,r} + \mu_m U_{m-1}}{p_1\mu_1 + \mu_m} + \left(\frac{p_1\mu_1}{p_1\mu_1+\mu_m}+4\right)\frac{\Delta}{C}\\
=&\frac{p_1\mu_1 - \left(p_1\mu_1 -\frac{\mu_m v_{m}}{v_{1}}+p_1\mu_1\sum_{r=m+1}^M \frac{v_r}{v_1}\right) L_{1,1}+ \mu_m U_{m-1} }{p_1 \mu_1 + \mu_m} + \frac{c_m}{C}\Delta\\
=&\frac{p_1\mu_1}{p_1 \mu_1 + \mu_m} -
\left(\frac{p_1\mu_1}{p_1\mu_1+\mu_m}\left(1+\sum_{r=m+1}^M\frac{v_r}{v_1}\right) - \frac{\mu_m}{p_1\mu_1+\mu_m}\frac{v_m}{v_1}\right)L_{1,1} + \frac{\mu_m }{p_1 \mu_1 + \mu_m}U_{m-1} + \frac{c_m}{C}\Delta\\
=&1-a_m - b_mL_{1,1} + a_mU_{m-1}+ \frac{c_m}{C}\Delta\\
=&U_m
\end{align}
where $$c_m = \frac{5p_1\mu_1\sum_{r=m+1}^M (r-1)v_r}{p_1\mu_1+\mu_m} + \frac{5\sum_{r=2}^{m-1} \mu_r v_r - 5(m-2)\mu_m v_m}{p_1\mu_1+\mu_m}  + \frac{p_1\mu_1}{p_1\mu_1+\mu_m}+4.$$

\subsection{Proof of Lemma \ref{lem: L11 lift}: Convergence of  $L_{1,1}(n)$}
\loneonelift*

Starting from $L_{1,1}(n)\leq s_{1,1}^*-\frac{6\Delta}{C},$ we can apply Lemma \ref{eq: iter lower bound} to obtain lower bounds $L_{1,m}(n)$ for $m\geq 2$ and then apply Lemma \ref{eq: iter upper bound} to obtain upper bounds $U_m(n)$ for all $m\geq 2,$ including $U_M(n).$ Then from $U_M(n),$ we obtain new lower bound $L_{1,1}(n+1).$ This iterative process implies that $U_M(n)$ and $L_{1,1}(n+1)$ are both a function of $L_{1,1}(n),$ as shown below.
Recall that in Lemma \ref{lem: iter prob upper bound}, we obtained
$$U_{m} =  1-a_m - b_m L_{1,1} + a_mU_{m-1} + \frac{c_m\Delta}{C}.$$
By recursively substituting $U_m,$ we can write $U_{M}$ as a function of $L_{1,1}$ as follows: 
\begin{align*}
U_{M} = \sum_{m=2}^{M} (1-a_m) \prod_{j=m+1}^M a_j - L_{1,1} \sum_{m=2}^{M} b_m \prod_{j=m+1}^M a_j + \frac{\Delta}{C} \sum_{m=2}^{M} c_m\prod_{j=m+1}^M a_j 
\end{align*}
Let us consider 
\begin{align}
L_{1,1}(n+1) =& 1 - \frac{1-\xi}{2\mu_1N^\alpha} - U_{M}(n) - \frac{6\Delta}{C} \notag \\
            =& \prod_{m=2}^M a_m  + L_{1,1} \sum_{m=2}^{M} b_m  \prod_{j=m+1}^M a_j - \frac{\Delta}{C} \left(\sum_{m=2}^{M} c_m\prod_{j=m+1}^M a_j+6\right)-\frac{1-\xi}{2\mu_1N^\alpha}\notag 
\end{align}
where we use $\sum_{m=2}^{M} (1-a_m) \prod_{j=m+1}^M a_j = \sum_{m=2}^{M} (\prod_{j=m+1}^M a_j-\prod_{j=m}^M a_j)=1-\prod_{m=2}^M a_m.$
In Lemma \ref{lem:xi}, we will show $1-\xi = \mu_1\prod_{m=2}^M a_m.$ We now center $L_{1,1}$ around $ \frac{1}{\mu_1} - \frac{C_M\Delta}{C(1-\xi)}-\frac{1}{2\mu_1N^\alpha}$ and have
\begin{align*}
    L_{1,1}(n+1) - \frac{1}{\mu_1} + \frac{C_M\Delta}{C(1-\xi)}+\frac{1}{2\mu_1N^\alpha}= \xi\left(L_{1,1}(n) - \frac{1}{\mu_1} + \frac{C_M\Delta}{C(1-\xi)}+\frac{1}{2\mu_1N^\alpha}\right), 
\end{align*}
where $\xi = \sum_{m=2}^{M} b_m  \prod_{j=m+1}^M a_j$ and $C_M = \sum_{m=2}^{M} c_m\prod_{j=m+1}^M a_j+6.$ 

Next we study the probability of $\epsilon_1(n+1)$ given $\epsilon_1(n).$ From Lemma \ref{lem: iter prob lower bound}, we have 
\begin{align*}
\epsilon_{m}=& e^{-\frac{v_m^2\log^2 N}{C^2}} + \left( \frac{C}{v_m\Delta}+1\right)\epsilon_{m-1} \\
\leq& e^{-\frac{\bar v^2\log^2 N}{C^2}} + \left( \frac{C}{\bar v \Delta}+1\right)\epsilon_{m-1}
\end{align*}
By expanding the above inequality from $\epsilon_M$ until $\epsilon_1,$ it implies that 
\begin{align*}
\epsilon_{M} \leq& \sum_{m=2}^M e^{-\frac{\bar v^2\log^2 N}{C^2}} \left( \frac{C}{\bar v\Delta}+1\right)^{M-m} + \epsilon_1 \left(\frac{C}{\bar v \Delta}+1\right)^{M-1}\\
\leq& \left(\sum_{m=2}^M e^{-\frac{\bar v^2\log^2 N}{C^2}} + \epsilon_1\right)\left(\frac{C}{\bar v \Delta}+1\right)^{M-1}\\
\leq& M \epsilon_1 \left(\frac{C}{\bar v \Delta}+1\right)^{M-1}
\end{align*} 
where the inequality holds because $\epsilon_1 \geq e^{-\frac{\bar v^2\log^2 N}{C^2}}.$
From Lemma \ref{lem: iter prob upper bound}, we have 
\begin{align*}
\sigma_{m}=& e^{-\frac{\log^2 N}{C^2}} +\left(\frac{C}{\Delta}+1\right) \left(\sigma_{m-1} + \sum_{m=1}^M \epsilon_m\right)\\
\leq& e^{-\frac{\log^2 N}{C^2}} + M\epsilon_M + \left(\frac{C}{\Delta}+1\right) \sigma_{m-1}
\end{align*} which implies that
\begin{align*}
\sigma_{M}
\leq& \left(e^{-\frac{\log^2 N}{C^2}} + M\epsilon_M\right) \left(\frac{C}{\Delta}+1\right)^M \\
\leq& \left(e^{-\frac{\log^2 N}{C^2}} + M^2\epsilon_1\left(\frac{C}{\bar v \Delta}+1\right)^{M} \right) \left(\frac{C}{\Delta}+1\right)^{M-1} \\
\leq& \epsilon_1 (M^2+1) \left(\frac{C}{\bar v \Delta}+1\right)^{2M-1}
\end{align*}
Therefore, we have 
\begin{align*}
e^{-\frac{\log^2 N}{C^2}} + \left(\frac{C}{\Delta}+1\right) \sigma_{M} \leq& e^{-\frac{\log^2 N}{C^2}} + \epsilon_1 (M^2+1) \left(\frac{C}{\bar v\Delta}+1\right)^{2M}\\
\leq& \epsilon_1 (M^2+2) \left(\frac{C}{\bar v\Delta}+1\right)^{2M} = \epsilon_1(n+1)
\end{align*}
Lastly, we prove the ``monontocity'' improvement of $\{L_{1,1}(n)\}_n$ by studying 
\begin{align}
   L_{1,1}(n+1) - L_{1,1}(n) = (1-\xi) \left(\frac{1}{\mu_1} - \frac{C_M\Delta}{C(1-\xi)}-\frac{1}{2\mu_1N^\alpha}-L_{1,1}(n)\right), \nonumber
\end{align} which is positive for $L_{1,1}(n) < \frac{\lambda}{\mu_1} - \frac{6\Delta}{C} < \frac{1}{\mu_1} - \frac{C_M\Delta}{C(1-\xi)}-\frac{1}{2\mu_1N^\alpha}.$

\section{Proof of Lemma \ref{lem:gen diff}}\label{app:gen-diff}
According to the definition of $e_{j,m}$ and $f(s)$ in \eqref{eq:steinsolution}, we have $$f(s+e_{j,m})=g\left(\sum_{i=1}^b\sum_{m=1}^M s_{i,m}+\frac{1}{N}\right)$$ and
$$f(s-e_{j,m})=g\left(\sum_{i=1}^b\sum_{m=1}^M s_{i,m}-\frac{1}{N}\right)$$for any $1\leq j \leq b.$ Therefore,
\begin{align*}
    &Gg\left(\sum_{i=1}^b\sum_{m=1}^M s_{i,m}\right) \\
    =&
    N \lambda\left(1-A_b(S)\right)\left(g\left(\sum_{i=1}^b\sum_{m=1}^M s_{i,m}+\frac{1}{N}\right)-g\left(\sum_{i=1}^b\sum_{m=1}^M s_{i,m}\right)\right)\\
    &+N \left(\sum_{m=1}^M(1-p_m)\mu_m s_{1,m}\right)\left(g\left(\sum_{i=1}^b\sum_{m=1}^M s_{i,m}-\frac{1}{N}\right)-g\left(\sum_{i=1}^b\sum_{m=1}^M s_{i,m}\right)\right),
\end{align*}
where the first term represents the transitions when a job arrives and the second term represents the transitions when a job departures from the system.
Note $(1-p_m)\mu_m s_{1,m}$ is the rates at which jobs leave the system when in phase $m$ in the state $s$. Therefore, $\sum_{m=1}^M(1-p_m)\mu_m s_{1,m}$ is the total departure rate. 
Define $d_1=\sum_{m=1}^M(1-p_m)\mu_m s_{1,m}$ and its stochastic correspondence $D_1=\sum_{m=1}^M(1-p_m)\mu_m S_{1,m}$ for simple notations.  

Substituting the generator equation above to \eqref{eq:gencou}, we have \begin{align}
&\mathbb E\left[h\left(\sum_{i=1}^b\sum_{m=1}^M S_{i,m}\right)\right]\nonumber\\
=&\mathbb E\left[g'\left(\sum_{i=1}^b\sum_{m=1}^M S_{i,m}\right)\left(-\Delta\right) \right. \nonumber\\
&\left. -N\lambda(1-A_b(S))\left(g\left(\sum_{i=1}^b\sum_{m=1}^M S_{i,m}+\frac{1}{N}\right)-g\left(\sum_{i=1}^b\sum_{m=1}^M S_{i,m}\right)\right)\nonumber\right.\\
&\left.-N D_1\left(g\left(\sum_{i=1}^b\sum_{m=1}^M S_{i,m}-\frac{1}{N}\right)-g\left(\sum_{i=1}^b\sum_{m=1}^M S_{i,m}\right)\right)\right]. \label{gen-diff-expand}
\end{align}

According to \eqref{Gen:L}, it is easy to verify 
\begin{equation*}
g(x) = g'\left(x\right)=0.
\end{equation*} Also note that when $x>\eta+\frac{1}{N},$
\begin{equation}g'(x)=-\frac{x-\eta}{\Delta},\end{equation} so for $x>\eta+\frac{1}{N},$
\begin{equation}g''(x)=-\frac{1}{\Delta}.\end{equation}

By using mean-value theorem in the region $\mathcal T_1 = \{x ~|~ \eta-\frac{1}{N} \leq x \leq \eta+\frac{1}{N}\}$ and Taylor theorem in the region $\mathcal T_2 = \{ x ~|~ x > \eta+\frac{1}{N}\},$ we have
\begin{align}
g(x+\frac{1}{N})-g\left(x\right) =& \left(g\left(x+\frac{1}{N}\right)-g\left(x\right)\right) \left(\mathbb{I}_{x \in \mathcal T_1} + \mathbb{I}_{x \in \mathcal T_2} \right) \nonumber\\
=&  \frac{g'(\xi)}{N}\mathbb{I}_{x \in \mathcal T_1} + \left(\frac{g'(x)}{N} +  \frac{g''(\zeta)}{2N^2}\right) \mathbb{I}_{x \in \mathcal T_2} \label{g+1/N}\\
g(x-\frac{1}{N})-g\left(x\right) =& \left(g\left(x-\frac{1}{N}\right)-g\left(x\right)\right) \left(\mathbb{I}_{x \in \mathcal T_1} + \mathbb{I}_{x \in \mathcal T_2} \right) \nonumber\\
=&  -\frac{g'(\tilde{\xi})}{N}\mathbb{I}_{x \in \mathcal T_1} + \left(-\frac{g'(x)}{N} +  \frac{g''(\tilde{\zeta})}{2N^2}\right) \mathbb{I}_{x \in \mathcal T_2} \label{g-1/N}
\end{align}
where $\xi, \zeta \in (x,x+\frac{1}{N})$ and $\tilde{\xi}, \tilde{\zeta} \in (x-\frac{1}{N},x).$
Substitute \eqref{g+1/N} and \eqref{g-1/N} into the generator difference in \eqref{gen-diff-expand}, we have
\begin{align}
&\mathbb E\left[h\left(\sum_{i=1}^b  S_{i}\right)\right]
= J_1 + J_2 + J_3,
\end{align}
with
\begin{align}
J_1 =& \mathbb E\left[g'\left(\sum_{i=1}^b S_{i}\right)\left(\lambda A_b(S) - \lambda-\Delta+D_1\right)\mathbb{I}_{\sum_{i=1}^b S_{i} \in \mathcal T_2}\right], \label{G-expansion-SSC-r}\\
J_2=&\mathbb E\left[\left(g'\left(\sum_{i=1}^b S_{i}\right)\left(-\frac{\log N}{\sqrt{N}}\right)-\lambda(1-A_b(S))g'(\xi)+D_1g'(\tilde{\xi}) \right)\mathbb{I}_{\sum_{i=1}^b S_{i} \in \mathcal T_1} \right], \label{G-expansion-Gradient-1-r}\\
J_3=&-\mathbb E\left[\frac{1}{2N}\left(\lambda(1-A_b(S))g''(\zeta) + D_1g''(\tilde{\zeta})\right)\mathbb{I}_{\sum_{i=1}^b S_{i} \in \mathcal T_2}\right]. \label{G-expansion-Gradient-2-r}
\end{align}
Note that in \eqref{G-expansion-Gradient-1-r} and \eqref{G-expansion-Gradient-2-r}, we have that
$$\xi,\zeta  \in \left(\sum_{i=1}^b S_i,\sum_{i=1}^b S_i+\frac{1}{N}\right) ~\text{and}~ \tilde{\xi},\tilde{\zeta}\in \left(\sum_{i=1}^b S_i-\frac{1}{N},\sum_{i=1}^b S_i\right)$$ are random variables whose values depend on $\sum_{i=1}^b S_i.$ We do not include $\sum_{i=1}^b S_i$ in the notation for simplicity. The proof of Lemma \ref{lem:gen diff} is completed by upper bounding $J_2$ and $J_3,$ 
for which, we establish gradient bounds on $g'$ and $g''$ in Lemma \ref{lemma:g'} and Lemma \ref{lemma:g''}.

\begin{lemma}\label{lemma:g'}
Given  $x\in\left[ \eta  -\frac{2}{N}, \eta  +\frac{2}{N}\right],$ we have $$|g'(x)|\leq \frac{2}{\sqrt{N}\log N}.$$ \qed 
\end{lemma}
\begin{proof}
From the definition of $g$ function in \eqref{Gen:L}, we have
$$g'(x)=\frac{\max\left\{x-\eta, 0\right\}}{-\frac{\log N}{\sqrt{N}}}.$$
Hence, for any $x\in\left[ \eta  -\frac{2}{N}, \eta  +\frac{2}{N}\right],$ we have
$$|g'(x)| \leq \frac{|x-\eta|}{\frac{\log N}{\sqrt{N}}}\leq \frac{\frac{2}{N}}{\frac{\log N}{\sqrt{N}}}=\frac{2}{\sqrt{N}\log N}.$$
\end{proof}

\begin{lemma}\label{lemma:g''}
For $x>\eta,$ we have 
\begin{align*}
|g''(x)|\leq \frac{\sqrt{N}}{\log N}.  
\end{align*}
\qed 
\end{lemma}
\begin{proof}
From the definition of $g$ function in \eqref{Gen:L}, we have
$$g'(x)=\frac{\max\left\{x-\eta, 0\right\}}{-\frac{\log N}{\sqrt{N}}}.$$
For $x> \eta,$ we have
$$g'(x)=\frac{x-\eta}{-\frac{\log N}{\sqrt{N}}},$$ which implies
\begin{eqnarray*}
|g''(x)|=\left|\frac{1}{-\frac{\log N}{\sqrt{N}}}\right|=\frac{\sqrt{N}}{\log{N}}.
\end{eqnarray*}
\end{proof}
Based on gradient bounds in Lemma \ref{lemma:g'} and \ref{lemma:g''} and note $\sum_{m}(1-p_m)\mu_m s_{1,m} \leq \mu_{\max} s_{1} \leq \mu_{\max},$ then we have
\begin{align*}
J_2 + J_3 
\leq &\mathbb E\left[\left(g'\left(\sum_{i=1}^{b} S_i\right)\left(-\frac{\log N}{\sqrt{N}}\right)+\lambda|g'(\xi)|+\mu_{\max}|g'(\tilde{\xi})|\right)\mathbb{I}_{\sum_{i=1}^{b} S_i \in \mathcal T_1}\right] \\
&+ \mathbb E\left[\frac{1}{N}\left(\lambda |g''(\eta)|  + \mu_{\max} |g''(\tilde{\eta})| \right) \mathbb{I}_{\sum_{i=1}^{b} S_i \in \mathcal T_2}\right] \\
\leq & \frac{4\mu_{\max}}{\sqrt{N}\log N} + \frac{\lambda+\mu_{\max}}{N} \frac{\sqrt{N}}{\log N} \\
= & \frac{5\mu_{\max}+\lambda}{\sqrt{N}\log N}
\end{align*}

\def \SSCsonestwonegative{\ref{SSC:s1 s2 negative}}

\section{Lemma \SSCsonestwonegative~ and the Proof}
\label{app:inssc}

\begin{restatable}{lemma}{SSCnegative}\label{SSC:s1 s2 negative}
For any $s \in \mathcal S_{ssp_1},$ 
\begin{align*} 
\left(\lambda+\Delta-\sum_{m=1}^M(1-p_m)\mu_m s_{1,m}\right)\mathbb{I}_{\sum_{i=1}^bs_i> \lambda +k\Delta+\frac{1}{N}} \leq 0.
\end{align*} \hfill{$\square$}
\end{restatable}

\begin{proof}
We consider the following linear programming problem $$\min_{s_{1,m} \in \mathcal S_{ssp_1}} \sum_{m=1}^M(1-p_m)\mu_ms_{1,m},$$ with $\mathcal S_{ssp_1}$ defined by
\begin{align*}
\mathcal S_{ssp_1} = \left\{s ~|~ s_1 \geq \lambda + \left(k-\zeta-6\right)\Delta, ~s_{1,m} \geq s_{1,m}^* - \theta_m\Delta\right\}.
\end{align*} 
Recall $w_m = (1-p_m)\mu_m.$ The minimum value is achieved when the maximum mass is allocated to $m^*$ such that $w_{m^*} = w_l = \min_m w_m.$ Therefore, we have
\begin{align*}
\sum_{m=1}^M w_m s_{1,m}
\stackrel{(a)}{\geq}& \sum_{m\neq m^*}^M w_m (s_{1,m}^* - \theta_m \Delta) + w_{m^*} \left(s_{1,m^*}^* + \left(k-\zeta-6 + \sum_{m\neq m^*}^M\theta_m \right)\Delta\right)\\
\stackrel{(b)}{=}& \lambda + w_l\left(k-\zeta-6+\sum_m \theta_m\right)\Delta -\sum_m w_m\theta_m \Delta\\
\stackrel{(c)}{=}& \lambda + \Delta 
\end{align*} 
where
\begin{itemize}
\item $(a)$ holds because $s_1 \geq \lambda+(k-\zeta-6)\Delta$ and $s_{1,m}, \forall m \neq m^*$ takes $L_{1,m}=s_{1,m}^* - \theta_m \Delta;$
\item $(b)$ holds because $\sum_m w_m s_{1,m}^* = \lambda;$
\item and $(c)$ holds because $w_l(k-\zeta-6+\sum_m \theta_m) -\sum_m w_m\theta_m = 1$ given carefully chosen $\zeta = \frac{4w_ub}{w_l}[(\frac{1}{w_l}-\frac{1}{w_u})\sum_m \theta_m w_m + \frac{1}{w_l}+6]$ and $k = \frac{\sum_{m}\theta_m w_m}{w_u}+(1+\frac{w_l}{4w_ub})\zeta - \sum_m \theta_m.$
\end{itemize}
\end{proof}

\section{Lemma \ref{SSC:out} and the Proof}

\begin{restatable}{lemma}{sscout}\label{SSC:out}
For a large $N$ such that , we have
$$\mathbb P\left( S \notin \mathcal S_{ssp}\right) \leq \frac{2}{N^{2}}.$$  \hfill{$\square$}
\end{restatable}

\begin{proof}
The proof of Lemma \ref{SSC:out} again relies on iterative state space peeling, which is based on Theorem \ThmmainA~ and Lemma \ref{SSC:s1 s2} below.

\begin{lemma}[A Lower Bound on $S_1$ via $\sum_{i=2}^b S_i$]\label{SSC:s1 s2}
\begin{align*}
&\mathbb P\left(\min\left\{\lambda +k\Delta-S_1,\sum_{i=2}^{b} S_i\right\} \leq (\zeta+6)\Delta\right) \geq 1-\frac{1}{N^2},\end{align*} 
where $\zeta = \frac{4w_ub}{w_l}[(\frac{1}{w_l}-\frac{1}{w_u})\sum_m w_m \theta_m + \frac{1}{w_l}+6]$ and $k = \frac{\sum_{m}w_m\theta_m}{w_u}+(1+\frac{w_l}{4bw_u})\zeta - \sum_m \theta_m.$  \qed
\end{lemma}

Based on Theorem \ref{Thm:equilibrium} and Lemma \ref{SSC:s1 s2}, we define sets $\tilde{\mathcal S}_1$ and $\tilde{\mathcal S}_2$ such that 
\begin{align}
\tilde{\mathcal S}_1=&\left\{s~|~s_{1,m} \geq s_{1,m}^* - \theta_m\Delta\right\}\\
\tilde{\mathcal S}_2=&\left\{s~| \min\left\{\eta-s_1,\sum_{i=2}^{b} s_i\right\} \leq (\zeta+6)\Delta\right\}.
\end{align}
According to the union bound and Theorem \ref{Thm:equilibrium} and Lemma \ref{SSC:s1 s2}, we have 
\begin{align*}
\mathbb P \left( S \notin \tilde{\mathcal S}_1\cap \tilde{\mathcal S}_2\right)\leq \frac{M}{N^{3}} + \frac{1}{N^{2}}\leq \frac{2}{N^{2}}. 
\end{align*}
We note that  $\tilde{\mathcal S}_1\cap \tilde{\mathcal S}_2 $  is a subset of  ${\mathcal S}_{ssp}.$ This is because for any $s$ which satisfies 
$$\min\left\{\lambda+k\Delta-s_1,\sum_{i=2}^{b} s_i\right\} \leq (\zeta+6)\Delta,$$ we either have $$\lambda+k\Delta-s_1\leq (\zeta+6)\Delta,$$ which implies 
$$s_1\geq \lambda+\left(k-\zeta-6\right)\Delta;$$ or 
$$\sum_{i=2}^{b} s_i\leq \eta-s_1,$$ which implies 
$$\sum_{i=1}^{b} s_i\leq \eta.$$ Note that $$\tilde{\mathcal S}_1\cap\left\{s\left|~s_1\geq \lambda+\left(k-\zeta-6\right)\Delta\right.\right\}={\mathcal S}_{ssp_1}$$ and 
$$\tilde{\mathcal S}_1\cap\left\{s~\left|~\sum_{i=1}^{b} s_i\leq \eta\right.\right\}\subseteq {\mathcal S}_{ssp_2}.$$ We, therefore, have  $$\tilde{\mathcal S}_1\cap \tilde{\mathcal S}_2\subseteq {\mathcal S}_{ssp},$$ and \begin{align*}
\mathbb P\left( S \notin {\mathcal S}_{ssp}\right) \leq  \mathbb P\left( S \notin \tilde{\mathcal S}_1\cap \tilde{\mathcal S}_2\right)
\leq \frac{2}{N^{2}},
\end{align*} so Lemma \ref{SSC:out} holds.
\end{proof}

Next, we prove Lemma \ref{SSC:s1 s2}.
\subsection{Proof of Lemma \ref{SSC:s1 s2}}
Recall $w_u = \max_{1\leq m \leq M} (1-p_m)\mu_m,$ $w_l = \min_{1\leq m \leq M} (1-p_m)\mu_m,$ and $L_{m} = s_{1,m}^* - \theta_m \Delta.$ Let $\zeta = \frac{4w_ub}{w_l}\left((\frac{1}{w_l}-\frac{1}{w_u})\sum_m w_m \theta_m + \frac{1}{w_l}+6\right)$ and $k = \frac{\sum_{m}w_m\theta_m}{w_u}+(1+\frac{w_l}{4bw_u})\zeta - \sum_m \theta_m.$ 

Consider Lyapunov function 
\begin{align}
V(s) = \min\left\{\lambda +k\Delta-s_1, \sum_{i=2}^b s_i\right\} \label{Ly:SSC}
\end{align} and define  
$$\mathcal E = \left\{s ~|~ s_{1,m} \geq L_{m}, \forall 1\leq m \leq M\right\}.$$ 

When $V(s) \geq \zeta\Delta,$ the following two inequalities hold
\begin{align}
             &s_1 \leq \lambda + (k-\zeta)\Delta \leq 1 - \frac{1-\xi}{2\mu_1 N^\alpha}, \label{SSCp=0:s1}\\
&\sum_{i=2}^b s_i \geq \zeta\Delta. \label{SSCp=0:s2}
\end{align}
We have two observations based on \eqref{SSCp=0:s1} and \eqref{SSCp=0:s2}:
\begin{itemize}
\item \eqref{SSCp=0:s1} implies that $A_1(s) \leq \frac{1}{\sqrt{N}}$ under any policy in $\Pi$;  
\item \eqref{SSCp=0:s2} implies that $s_2 \geq \frac{\zeta\Delta}{b}$ because  $s_2 \geq s_3 \geq \cdots \geq s_b,$ and we have
\end{itemize}
\begin{align}
\sum_{m=1}^{M}(1-p_m)\mu_m s_{2,m} \geq w_l s_2 \geq \frac{w_lc_1\Delta}{b}, \label{p=0s21-s22-lowerbound}
\end{align}
where a finite buffer size is required such that the lower bound $w_l s_2 \geq \frac{w_lc_1\Delta}{b}$ is meaningful.

We next study the Lyapunov drift when $V(s) \geq \zeta\Delta$ and $s\in \mathcal E$ by considering two cases:
\begin{itemize}
\item Suppose $\lambda +k\Delta-s_1 \geq \sum_{i=2}^b s_i \geq \zeta\Delta.$ In this case, $V(s) = \sum_{i=2}^b s_i$ and
\begin{align}
\nabla V(s) \leq& \lambda (A_1(s)-A_b(s))- \sum_{m=1}^M(1- p_m)\mu_m s_{2,m}\label{SSCp=0-case1-1}\\
\stackrel{(a)}{\leq}& \frac{1}{\sqrt{N}}- \sum_{m=1}^M(1- p_m)\mu_m s_{2,m}\nonumber\\
\stackrel{(b)}{\leq}&  \frac{1}{\sqrt{N}}-\frac{w_lc_1\Delta}{b} \nonumber\\
\stackrel{(c)}{\leq}&  -\frac{w_lc_1\Delta}{2b} \nonumber
\end{align}  where
\begin{itemize}
\item $(a)$ holds because $A_1(s) \leq \frac{1}{\sqrt{N}}$ under any policy in $\Pi;$
\item $(b)$ holds because of \eqref{p=0s21-s22-lowerbound};
\item and $(c)$ holds because $\log N \geq \frac{4b}{w_lc_1}.$
\end{itemize}

\item Suppose $\sum_{i=2}^b s_i > \lambda +k\Delta-s_1 \geq \zeta\Delta.$ In this case, $V(s) = \lambda +k\Delta-s_1$ and
\begin{align}
\nabla V(s) 
\leq& -\lambda (1-A_1(s)) + \sum_{m=1}^M(1- p_m)\mu_m s_{1,m} - \sum_{m=1}^M(1- p_m)\mu_m s_{2,m}\label{SSCp=0-case2-1}\\ 
\leq& \frac{1}{\sqrt{N}}-\lambda + w_u s_1 - \sum_{m=1}^M\left(w_u-w_m\right) s_{1,m} - \sum_{m=1}^Mw_m s_{2,m}\nonumber\\ 
\stackrel{(a)}{\leq}& \frac{1}{\sqrt{N}} -\lambda + w_u \left(s_1-\sum_{m=1}^M L_{m}\right) +  \sum_{m=1}^Mw_m L_{m}-  \sum_{m=1}^Mw_m s_{2,m}\nonumber\\
\stackrel{(b)}{=}& \frac{1}{\sqrt{N}} + \left(w_u \left(k-\zeta + \sum_{m=1}^M \theta_m\right) -  \sum_{m=1}^Mw_m\theta_m \right)\Delta-  \sum_{m=1}^Mw_m\mu_m s_{2,m}\nonumber\\
\stackrel{(c)}{\leq}& \frac{1}{\sqrt{N}} + \left(w_u \left(k-\zeta + \sum_{m=1}^M \theta_m\right) -  \sum_{m=1}^Mw_m\theta_m \right)\Delta -\frac{w_lc_1\Delta}{b}\nonumber\\
{=}& \frac{1}{\sqrt{N}} -\frac{3w_lc_1\Delta}{4b} \\
\stackrel{(d)}\leq& -\frac{w_lc_1\Delta}{2b} \nonumber
\end{align} where
\begin{itemize}
\item $(a)$ holds because $s_{1,m} \geq L_m, \forall m \geq 1;$ 
\item $(b)$ holds because $s_1 \leq \lambda + (k-\zeta)\Delta$ and $L_m = s_{1,m}^* - \theta_m\Delta, \forall m \geq 1;$
\item $(c)$ holds because $w_u (k-\zeta + \sum_{m=1}^M \theta_m) - \sum_{m=1}^Mw_m\theta_m =-\frac{w_lc_1}{4b}$ given $k$ and $\zeta;$
\item and $(d)$ holds because $\log N \geq \frac{4b}{w_lc_1}.$
\end{itemize}
\end{itemize}

Next, we further show $\nabla V(s) \leq w_{u}$ based on the upper bounds \eqref{SSCp=0-case1-1} and \eqref{SSCp=0-case2-1}.
\begin{itemize}
\item Consider the upper bound in \eqref{SSCp=0-case1-1}. We have
\begin{align}
\nabla V(s) \leq \lambda (A_1(s)-A_b(s)) - \sum_{m=1}^M(1-p_m)\mu_m s_{2,m}\leq 1 \leq w_{u}, \nonumber
\end{align}
where $1 \leq w_{u}$ holds because $\sum_{m=1}^Mv_m=1.$
\item Consider the upper bound in \eqref{SSCp=0-case2-1}. We have
\begin{align*}
\nabla V(s) \leq& -\lambda (1-A_1(s)) + \sum_{m=1}^M(1-p_m)\mu_m s_{1,m} -\sum_{m=1}^M(1-p_m)\mu_m s_{2,m} \\
\leq& \sum_{m=1}^M(1-p_m)\mu_m s_{1,m} \leq w_{u}, 
\end{align*}
where the last inequality holds because $\sum_m s_{1,m} = s_1 \leq 1.$
\end{itemize}

We now apply Lemma \ref{tail-bound-cond}. Define  $B = \zeta\Delta,$ $\gamma = \frac{w_l \zeta\Delta}{2b}$ and $\delta = w_u.$ Combining $q_{\max} =w_{u}N$ and $\nu_{\max} = \frac{1}{N},$ we have $$\alpha = \frac{w_u}{w_u + \frac{w_l \zeta\Delta}{2b}} ~~\text{ and }~~ \beta = \frac{2w_ub}{w_l \zeta\Delta} + 1.$$ 
Choosing $j=3\sqrt{N}\log N,$ we have
\begin{align*}
\mathbb P\left(V(S) \geq B + 2 \nu_{\max} j\right) 
\stackrel{(a)}{=}& \mathbb P\left(V(S) \geq (\zeta+6)\Delta\right)\\
\stackrel{(b)}{\leq}& \left(\frac{1}{1 + \frac{w_l \zeta\Delta}{2w_ub}}\right)^{3\sqrt{N}\log N} + \beta \mathbb P\left(S \notin \mathcal E\right)\\
{\leq}& \left(1 - \frac{w_lc_1\Delta}{3w_ub}\right)^{3\sqrt{N}\log N} + \beta \mathbb P\left(S \notin \mathcal E\right)\\
\stackrel{(d)}{\leq}& e^{-\frac{w_lc_1}{w_ub}\log^2 N} + \left(\frac{2w_ub}{w_l \zeta\Delta} + 1\right) \mathbb P\left(S \notin \mathcal E\right)\\
\stackrel{(d)}{\leq}& e^{-\frac{w_lc_1}{w_ub}\log^2 N} + \left(\frac{2w_ub}{w_l \zeta\Delta} + 1\right)  \frac{M}{N^3}\\
\stackrel{(e)}{\leq}& \frac{1}{N^2}
\end{align*}  where
\begin{itemize}
\item $(a)$ holds by substituting $B = \zeta\Delta,$ $\nu_{\max}=\frac{1}{N}$ and $j=3\sqrt{N}\log N;$
\item $(b)$ holds based on Lemma \ref{tail-bound-cond};
\item $(c)$ holds because $\frac{w_lc_1}{w_ub}\leq \frac{1}{\Delta};$
\item $(d)$ holds by union bounds on $\mathbb P\left(S \notin \mathcal E\right);$
\item and $(e)$ holds because $\frac{w_lc_1}{w_ub}\geq 24$ and $2\left(\frac{2Mbw_u}{w_l \zeta\Delta} + M\right) \leq N.$
\end{itemize}

\end{document}